\documentclass[pdflatex,sn-mathphys-num]{sn-jnl}

\usepackage{graphicx}%
\usepackage{multirow}%
\usepackage{amsmath,amssymb,amsfonts}%
\usepackage{amsthm}%
\usepackage{mathrsfs}%
\usepackage[title]{appendix}%
\usepackage{xcolor}%
\usepackage{textcomp}%
\usepackage{manyfoot}%
\usepackage{booktabs}%
\usepackage{algorithm}%
\usepackage{algorithmicx}%
\usepackage{algpseudocode}%
\usepackage{listings}%

\usepackage{mathtools,array,microtype,needspace,tikz}
\hypersetup{hidelinks,bookmarksdepth=3,pdftitle={Superlinear Convergence of Cyclic Gradient Methods with Historical Ritz Step Sizes},pdfauthor={Ran Gu}}

\theoremstyle{thmstyleone}
\newtheorem{theorem}{Theorem}[section]
\newtheorem{lemma}[theorem]{Lemma}
\theoremstyle{thmstyletwo}
\newtheorem{remark}[theorem]{Remark}
\numberwithin{equation}{section}

\newcommand{\R}{\mathbb R}
\newcommand{\K}{\mathcal K}
\newcommand{\vol}{\operatorname{vol}}
\newcommand{\diag}{\operatorname{diag}}
\newcommand{\spann}{\operatorname{span}}

\newcommand{\algSCR}{\textnormal{SCR}}
\newcommand{\algRCR}{\textnormal{RCR}}
\newcommand{\sig}{\sigma_\ell}

\begin{document}

\title[Superlinear Convergence with Historical Ritz Steps]{Superlinear Convergence of Cyclic Gradient Methods with Historical Ritz Step Sizes}

\author*[1]{\fnm{Ran} \sur{Gu}}\email{rgu@nankai.edu.cn}
\affil*[1]{\orgdiv{NITFID, School of Statistics and Data Science, LPMC, KLMDASR, LEBPS and AAIS}, \orgname{Nankai University}, \orgaddress{\city{Tianjin}, \postcode{300071}, \country{China}}}

\abstract{The long-step Barzilai--Borwein (BB1) method, cyclic BB1, and limited memory steepest descent (LMSD) are special cases of limited memory steepest descent with retards (LMSDR), a family of methods using historical Ritz step sizes. We analyze a cyclic Ritz-step rule within LMSDR and a sampled LMSDR variant on strictly convex quadratics, seeking R-superlinear convergence guarantees in higher spectral dimensions. Let $n$ be the problem dimension and $m$ the history length. In exact arithmetic, we begin at $n=m+1$, the first dimension beyond the finite-termination case, and obtain an exact doubly exponential rate per gradient update for rolling reuse outside an explicit, possibly empty exceptional set. We extend this rate to nonempty open sets for $m+1<n\le2m$ when each Ritz step is reused $\ell\ge4$ times. To go beyond this range, we rebuild the history with $m$ sampling updates at each cycle. With positive sampling steps $\tau$ chosen in advance and factors $1-\tau\lambda_i$ uniformly bounded in modulus and away from zero, this gives almost-everywhere doubly exponential convergence when $(2\ell-1)^m>\binom{n-1}{m}$. For any $n>m$, increasing $\ell$ eventually satisfies this condition. Geometrically spaced choices among finitely many sampling lists preserve the guarantee.} 

\maketitle

\enlargethispage{2\baselineskip}
\noindent This work was supported by the National Key R\&D Program of China (grant 2022YFA1003800), the National Natural Science Foundation of China (grant 12671364), and the Natural Science Foundation of Tianjin (grant 25JCJQJC00300).

\section{Introduction}

\subsection{Spectral models and step-size selection}

The Barzilai--Borwein (BB) method can converge R-superlinearly on a two-dimensional strictly convex quadratic, although each iteration only chooses a step size along the negative gradient~\citep{BB,Dai2013}. Already in three dimensions, the exact recurrence for its gradient components is difficult to analyze. \citet{DaiFletcher} studied a simplified model in which the component magnitudes are widely separated, allowing smaller terms, including bounded terms depending on the eigenvalues, to be omitted from the logarithmic recurrence. The simplified BB model predicts superlinear behavior in three dimensions and linear behavior in four dimensions, in agreement with their numerical observations. For cyclic steepest descent (CSD), which reuses an exact steepest descent step for $\ell$ updates, the model predicts superlinear behavior for $n<2\ell$ and linear behavior for $n\ge2\ell$.

\citet{GuDu} extended this analysis to the limited memory steepest descent (LMSD) method of \citet{Fletcher} and proposed a modified method (MLMSD). LMSD uses all $m$ reciprocal Ritz values once in each sweep, whereas MLMSD chooses one at each refresh. The reported superlinear model ranges are $n\le3m$ for LMSD and $n\le2m(m+1)-1$ for the oldest-gradient selection in MLMSD. For $m=1$, the latter bound gives dimension three, in agreement with the BB model. The constructions at the respective upper endpoints use particular arrangements of the component magnitudes. Thus the predicted range grows linearly with $m$ for LMSD and quadratically for MLMSD. In actual computations, \citet[Sections~5--6]{GuDu} report gains on 1000-variable quadratics, well beyond the analyzed model dimensions, and on a strictly convex nonquadratic problem of the same size.

A possible explanation for these practical gains is that superlinear convergence in a larger dimension helps reduce a larger group of dominant spectral components together. As these components decrease, another group can become dominant in the next Ritz calculation. Repeating this process could accelerate a problem with many more eigenvalues. Properties of gradient iterations in small dimensions have also guided step-size design, including methods with finite termination in two or three dimensions~\citep{HuangDaiLiuTermination,XieLiuSunYuanTermination}.

\subsection{Actual convergence and the role of past gradients}

A proof for the actual iterates must control the terms omitted by the simplified model. This adds a substantial difficulty: almost-everywhere R-superlinear convergence of three-dimensional BB1 remains open. Recent work has addressed the actual iterations directly. For each $n\ge4$, \citet{LiJiangHong2026} construct an open family of quadratics and initial points on which the BB1 gradient norm has geometric lower and upper bounds, so R-superlinear convergence fails. For two-dimensional CSD, \citet{ShenLiDaiCSD} derive a component-ratio recurrence and analyze superlinear decay at cycle boundaries. The equal-magnitude initial gradient components require a separate treatment: \citet{LiWangTwoSpectrum} identify them as the nonterminating linear case, while all other nonterminating initial points converge R-superlinearly when $\ell\ge2$. Explicit linear trajectories in two and three dimensions are given in \citet{LiWangCounter}. For arbitrary $n$, \citet{GuCSD} proves doubly exponential, hence R-superlinear, convergence for almost every initial point when $2\ell>n$, matching the superlinear range of the simplified CSD model.

The historical gradient in the BB step makes its component recurrence more difficult to analyze than the cycle recurrence of CSD. For a positive definite quadratic with Hessian $A$, write $\alpha_{\mathrm{SD}}(v)=v^Tv/(v^TAv)$ for the exact steepest descent step determined by a nonzero gradient $v$. In our notation, the BB update and CSD cycle relation are
\[
\begin{aligned}
 \mathrm{BB1}:\qquad
 g_{t+1}&=[I-\alpha_{\mathrm{SD}}(g_{t-1})A]g_t,\\
 \mathrm{CSD}:\qquad
 \bar g_{k+1}&=[I-\alpha_{\mathrm{SD}}(\bar g_k)A]^\ell\bar g_k,
 \qquad \bar g_k=g_{k\ell}.
\end{aligned}
\]
The BB step is defined by $g_{t-1}$ but applied to $g_t$. In CSD, the step determined by $\bar g_k$ is held fixed for all $\ell$ updates in the cycle, beginning at $\bar g_k$. Even the simplified three-dimensional BB recurrence requires a detailed analysis in \citet{DaiFletcher}; the terms retained by the actual iteration add to this difficulty.

We extend the analysis to LMSD, which also uses historical gradients and includes BB1 when $m=1$. Its Ritz values come from the Krylov space generated by the oldest gradient in the stored block. Before the new steps are applied, the previous sweep changes this gradient, so its polynomial appears in the recurrence. For LMSD, \citet{CurtisGuo} prove R-linear convergence under assumptions on the history rank and normalized conditioning. \citet{YangYuanLMSD} obtain global R-linear convergence for rank-compressed weighted LMSD without those assumptions.

Reusing a step over a cycle gives a simpler recurrence in CSD. This suggests studying cyclic reuse in LMSD as well. We therefore consider the LMSDR family introduced by \citet{ZouMagoules}. Each of the $m$ reciprocal Ritz values is used $\ell$ times, with a complete sweep in the last $m$ updates of each cycle; the earlier updates can be put in any order. We call this class of LMSDR orderings rolling cyclic Ritz reuse (\algRCR). With matched initialization and ordering, full-rank histories, and uninterrupted cycles, its special cases are
\[
\begin{array}{rcl}
\ell=1 &:& \text{LMSD},\\
m=1,\ \ell>1 &:& \text{cyclic BB1 (CBB1)},\\
m=\ell=1 &:& \text{BB1}.
\end{array}
\]
Related limited memory implementations and delayed-gradient methods are discussed by \citet{FerrandiHochstenbach} and \citet{ZouMagoulesReview}.

The component recurrences also relate this work to studies of gradient methods as dynamical systems. \citet{PronzatoWynnZhigljavsky} analyze limiting normalized weights and convergence rates for a family of gradient methods. \citet{FaberLiesenTichy} and \citet{FaberEtAl2026} study limiting directions and cycles in restarted Krylov methods. Our question concerns superlinear convergence with historical Ritz steps: we use the relative component sizes to estimate the accumulated decrease of each original gradient component, and determine how the dimension, memory length, and repetition count enter the resulting rate.

\subsection{Main results}

Our aim is to enlarge the range of spectral dimensions for which superlinear convergence can be proved. For RCR in dimension $n=m+1$, we determine an exact doubly exponential rate per gradient update for every $\ell$, outside an explicit exceptional set of initial weights. If $h_k$ is the oldest gradient defining the cycle-$k$ Ritz polynomial $p_k$, then, after initialization,
\[
 h_{k+1}=p_k(A)^{\ell-1}p_{k-1}(A)h_k.
\]
The factor $p_{k-1}(A)$ accounts for the updates between the defining and receiving gradients. When $n=m+1$, we can solve the equations for the Ritz polynomial explicitly. The vector $w_k$ of centered logarithmic weight ratios, defined in \eqref{eq:w}, then satisfies
\[
 w_{k+1}=(3-2\ell)w_k-2w_{k-1}.
\]
Growth of the two-step state $(w_k,w_{k-1})$ gives a lower bound on the total logarithmic decrease $-\sum_{j<k}\log|p_j(\lambda_i)|$ for every spectral component. This gives the rate for the gradient norm.

For $n>m+1$, we instead bound the values of the Ritz polynomials by separating the spectral components into groups. When $m+1<n\le2m$ and $\ell\ge4$, sufficient initial separation persists as the groups alternate in dominance, giving the same rate on nonempty open sets of initial points.

In sampled cyclic Ritz reuse (\algSCR), each cycle begins with a list of $m$ positive step sizes. The corresponding gradient updates provide the new history. For lists chosen in advance, the relative spectral weights at one cycle start determine those at the next. Suppose the factors $1-\tau\lambda_i$ are uniformly bounded in modulus and away from zero for all sampling steps $\tau$ and all active eigenvalues $\lambda_i$. Then the method converges doubly exponentially for almost every initial point when
\[
 (2\ell-1)^m>\binom{n-1}{m}.
\]
For fixed $m$ and any $n>m$, a sufficiently large $\ell$ satisfies this condition, as in CSD~\citep{GuCSD}. At cycles $0,1,2,4,8,\ldots$, we can also use the past iterates to choose a list from a prescribed finite set. Holding each choice until the next decision preserves the condition and the rate range.

Table~\ref{tab:scope} collects the model predictions, actual convergence results, and remaining open ranges.

\begin{table}[htbp]
\centering
\begingroup
\footnotesize
\setlength{\tabcolsep}{3pt} \renewcommand{\arraystretch}{1.05}
\begin{tabular}{@{}>{\raggedright\arraybackslash}p{0.20\linewidth}
                  >{\raggedright\arraybackslash}p{0.27\linewidth}
                  >{\raggedright\arraybackslash}p{0.49\linewidth}@{}}
\toprule
Algorithm and relation & Simplified-model analysis
& Actual iteration: results and remaining gap\\
\midrule
\textbf{BB1}\newline
$m=\ell=1$\newline
$1$ update
& $n\le3$: superlinear picture; $n\ge4$: linear
picture~\citep{DaiFletcher}.
& $n=2$: R-superlinear outside exceptional states~\citep{BB,Dai2013}.
$n=3$: almost-everywhere superlinearity open.
$n\ge4$: two-sided geometric gradient bounds on a constructed
open family~\citep{LiJiangHong2026}.\\
\midrule
\textbf{CSD}\newline
Current cycle-start gradient\newline
$\ell$ updates/cycle
& $n<2\ell$: superlinear; $n\ge2\ell$: linear
model picture~\citep{DaiFletcher}.
& $n<2\ell$: doubly exponential a.e.~\citep{GuCSD}.
\emph{Superlinear-side threshold matched.}
Exceptional linear trajectories occur~\citep{LiWangCounter};
no classification claimed for $n\ge2\ell$.\\
\midrule
\textbf{CBB1}\newline
LMSDR with $m=1$, $\ell>1$\newline
$\ell$ updates/cycle
& ---
& Local R-linear theory and explicit linear cycles at
$n=3$, $\ell=2$~\citep[Sections~2--3]{CBB}.
For the separate regime $n=2$, we obtain an exact doubly exponential
rate outside explicit exceptional weights (Theorem~\ref{thm:B1}).\\
\midrule
\textbf{LMSD}\newline
LMSDR with $\ell=1$\newline
$m$ updates/sweep
& Reported scale $n\le3m$; superlinear construction at
$n=3m$ uses three synchronized blocks~\citep{GuDu}.
& R-linear guarantees under respective assumptions and
rank-compression rules~\citep{CurtisGuo,YangYuanLMSD}.
Here, $n=m+1$: exact-rate doubly exponential decay a.e. for
uninterrupted full-rank sweeps ($\ell=1$, Theorem~\ref{thm:B1});
the full model range is not covered.\\
\midrule
\textbf{MLMSD}\newline
Single-candidate selection\newline
$1$ update/refresh
& Oldest-gradient selection: reported scale
$n\le2m(m+1)-1$, $m\ge2$; superlinear orbit constructed at
the upper endpoint~\citep{GuDu}.
& R-linear guarantee~\citep{GuDu}; the model superlinear range
remains unproved for actual iterations.\\
\midrule
\textbf{LMSDR}\newline
RCR rule analyzed here\newline
$m\ell$ updates/cycle
& ---
& Here, $n=m+1$: explicit doubly exponential rate a.e., each
$\ell\ge1$. For $m+1<n\le2m$, $\ell\ge4$: nonempty open
sets of initial points with the same rate formula, without a
general-initialization guarantee.\\
\midrule
\textbf{Sampled LMSDR variant}\newline
SCR here\newline
$m(\ell+1)$ updates/cycle
& ---
& Here, $(2\ell-1)^m>\binom{n-1}{m}$: doubly exponential decay a.e.
Finite sets of lists with $h_{\mathrm{seq}}+\log\rho<0$: all allowed history-dependent
selection rules on a common full-measure set.\\
\bottomrule
\end{tabular}
\endgroup
\caption{Model predictions and actual-iteration guarantees, grouped by algorithm and relation to the specified LMSDR rule. A.e.\ means almost every initial point.}
\label{tab:scope}
\smallskip
\begin{minipage}{\linewidth}
\footnotesize
\textit{Notes.} Throughout the table, $n$ denotes the problem dimension. The model column concerns simplified iterations, whereas the last column concerns actual iterations. A.e.\ means almost every initial point, and a dash means no model result is summarized. Update counts include initialization and sampling. For finite sampling lists, $\rho=\binom{n-1}{m}/(2\ell-1)^m$ and $h_{\mathrm{seq}}$ is the sequence-count growth rate defined in Section~\ref{sec:adaptiveA}.
\end{minipage}
\end{table}

The rest of this paper is organized as follows. Section~\ref{sec:algorithms} defines the two methods and the convergence rates used below. For the rolling method, Section~\ref{sec:codimone} gives the exact rate when $n=m+1$, and Section~\ref{sec:two} proves the same rate on open sets when $m+1<n\le2m$. Section~\ref{sec:A} proves almost-everywhere convergence of the sampled method, extends it to adaptive sampling, and gives a numerical example.

\section{Gradient histories and cyclic Ritz steps}\label{sec:algorithms}

\subsection{Quadratic model and historical Ritz extraction}\label{sec:spectralSetting}

The gradient iterations and Ritz values are invariant under translations and orthogonal changes of coordinates, and components with the same eigenvalue can be combined. We therefore consider, without loss of generality, the $n$-dimensional diagonal problem
\begin{equation}\label{eq:quad}
 f(x)=\tfrac12 x^TAx,\qquad
 A=\diag(\lambda_1,\ldots,\lambda_n),\qquad g_t=Ax_t.
\end{equation}
Unless a subscript specifies otherwise, $\|\cdot\|$ denotes the Euclidean vector norm or the induced matrix norm; $\|\cdot\|_F$ denotes the Frobenius norm. The index $t$ counts individual gradient updates, while $k$ below counts cycles. A gradient update with step size $\alpha_t>0$ satisfies
\begin{equation}\label{eq:update}
 x_{t+1}=x_t-\alpha_tg_t,\qquad
 g_{t+1}=(I-\alpha_tA)g_t.
\end{equation}
We assume $0<\lambda_1<\cdots<\lambda_n$ and work in exact arithmetic.

The memory length is a positive integer $m$, with $n>m$ except in the finite-termination discussion. All stored step sizes are positive. At a time $t\ge m$, form the consecutive historical matrices
\begin{equation}\label{eq:history}
 H_t=[g_{t-m},\ldots,g_{t-1}],\qquad
 Y_t=\left[\frac{g_s-g_{s+1}}{\alpha_s}\right]_{s=t-m}^{t-1}.
\end{equation}
Equation~\eqref{eq:update} gives $Y_t=AH_t$. If $H_t$ has full column rank, take its thin QR factorization $H_t=QR$, where $Q\in\R^{n\times m}$ has orthonormal columns and $R\in\R^{m\times m}$ is nonsingular and upper triangular. The historical Ritz values are the eigenvalues of
\begin{equation}\label{eq:small}
 B_t=Q^TY_tR^{-1}=Q^TAQ.
\end{equation}
The first expression computes the small matrix from stored gradient differences, without products of $A$ with additional vectors. The identity $B_t=Q^TAQ$ explains why its eigenvalues are Ritz values. For a vector $h$, define the Krylov subspace $\K_m(A,h):=\spann\{h,Ah,\ldots,A^{m-1}h\}$.

We will use the following rank property whenever all components of the defining gradient are nonzero.

\begin{lemma}[Rank of a consecutive gradient history]\label{lem:krylovRank}
Let $n>m$ and let $h\in\R^n$ have no zero component. Then $h,Ah,\ldots,A^{m-1}h$ are linearly independent. If $h=g_{t-m}$ and the intervening step sizes are nonzero, the matrix $H_t$ in \eqref{eq:history} has rank $m$ and spans $\K_m(A,h)$.
\end{lemma}
\begin{proof}\hspace{0.5em}
If $\sum_{j=0}^{m-1}d_jA^jh=0$, its $i$th component gives $h_i\sum_{j=0}^{m-1}d_j\lambda_i^j=0$. Since $h_i\ne0$, the polynomial $\sum_jd_j\xi^j$ vanishes at all $n>m$ distinct eigenvalues. Its degree is at most $m-1$, so every coefficient is zero.

The $j$th subsequent gradient is
\[
 g_{t-m+j}=\prod_{s=0}^{j-1}(I-\alpha_{t-m+s}A)h,\qquad 0\le j<m.
\]
Its coefficient of $A^jh$ is $(-1)^j\prod_{s=0}^{j-1}\alpha_{t-m+s}\ne0$. Thus $H_t=WC$ for $W=[h,Ah,\ldots,A^{m-1}h]$ and an invertible triangular matrix $C$. This proves both the rank and the equality of the two spans.
\end{proof}

The next lemma gives equations for the polynomial whose roots are the historical Ritz values.

\begin{lemma}\label{lem:history}
If the preceding $m$ step sizes are nonzero and the history has rank $m$, then $\spann H_t=\K_m(A,h),\qquad h=g_{t-m}$. For $u_i=h_i^2>0$ and $n>m$, the polynomial $\pi(\xi)=\det(\xi I_m-B_t)$ is the unique monic degree-$m$ polynomial satisfying
\begin{equation}\label{eq:orthogonal}
 \sum_{i=1}^n u_i\pi(\lambda_i)\lambda_i^a=0,
 \qquad a=0,\ldots,m-1.
\end{equation}
Its roots are distinct and lie in $(\lambda_1,\lambda_n)$.
\end{lemma}
\begin{proof}\hspace{0.5em}
The $j$th historical gradient is a degree-$j$ polynomial in $A$ applied to $h$, for $j=0,\ldots,m-1$. Its leading coefficient is $(-1)^j$ times the product of the preceding step sizes, namely $(-1)^j\prod_{s=0}^{j-1}\alpha_{t-m+s}$, which is nonzero. The change of basis between $H_t$ and $h,Ah,\ldots,A^{m-1}h$ is therefore triangular and invertible. This shows that $\spann H_t=\K_m(A,h)$. For full-support $h$ and $n>m$, these Krylov vectors are linearly independent by Lemma~\ref{lem:krylovRank}, so the history has rank $m$.

Let $W=[h,Ah,\ldots,A^{m-1}h]$ and $c=(W^TW)^{-1}W^TA^mh$. Then
\[
 (W^TW)^{-1}W^TAW=[e_2,\ldots,e_m,c],
\]
where $e_j$ are the coordinate vectors and $c=(c_0,\ldots,c_{m-1})^T$; when $m=1$, only the last column is present. This matrix is similar to $B_t=Q^TAQ$ because $H_t=QR$ and $W$ span the same subspace. Its characteristic polynomial is therefore $\pi(\xi)=\xi^m-\sum_{j=0}^{m-1}c_j\xi^j$. The definition of $c$ gives $W^T(A^mh-Wc)=0$, whose individual equations give
\[
 0=h^TA^a\pi(A)h
   =\sum_i u_i\lambda_i^a\pi(\lambda_i),\qquad 0\le a<m.
\]
Conversely, suppose that a monic degree-$m$ polynomial
\[
 \pi(\xi)=\xi^m-\sum_{j=0}^{m-1}c_j\xi^j
\]
satisfies \eqref{eq:orthogonal}. Since $\pi(A)h=A^mh-\sum_{j=0}^{m-1}c_jA^jh=A^mh-Wc$, these conditions can be written componentwise as
\[
 (A^ah)^T(A^mh-Wc)=0,\qquad a=0,\ldots,m-1.
\]
The $(a+1)$st column of $W$ is exactly $A^ah$, so these equations together give
\[
 W^T(A^mh-Wc)=0,
 \qquad\text{or equivalently}\qquad
 (W^TW)c=W^TA^mh.
\]
The coefficient matrix $W^TW$ is positive definite: if $d^TW^TWd=0$, then $Wd=0$, that is, $\sum_jd_jA^jh=0$. Since there are $n>m$ distinct eigenvalues and $h_i\ne0$ for every $i$, this implies $d=0$. Thus $c$ is unique, and so is the monic degree-$m$ polynomial satisfying \eqref{eq:orthogonal}.

The roots are distinct and lie in $(\lambda_1,\lambda_n)$. This is a standard root property of the monic degree-$m$ polynomial satisfying \eqref{eq:orthogonal}: its $m$ roots are real, distinct, and lie strictly between the smallest and largest support points; see \citep[Section~18]{NIST:DLMF}.
\end{proof}

We use the normalized Ritz polynomial
\begin{equation}\label{eq:filter}
 p(\xi)=\frac{\pi(\xi)}{\pi(0)}
     =\prod_{j=1}^m\left(1-\frac{\xi}{\theta_j}\right),
\end{equation}
The positive Ritz values $\theta_j$ are the roots of $\pi$, so $\pi(0)\ne0$. A complete sweep uses each reciprocal Ritz value $1/\theta_j$ as a step size once, multiplying the gradient by $p(A)$. The polynomial update keeps every active spectral component nonzero if $p(\lambda_i)\ne0$ for every active index $i$.

\begin{remark}[Rank and finite termination]\label{rem:rank}
If $n\le m$ and the history spans $\R^n$, its exact rank is $n$ and its Ritz values, computed after rank reduction, are $\lambda_1,\ldots,\lambda_n$. Using each reciprocal once gives finite termination. Exact rank-compressed sweeps are analyzed in \citet{YangYuanLMSD}. Below we study $n>m$ and call a vector full-support if all its components are nonzero. Lemma~\ref{lem:krylovRank} then gives rank $m$ for its history. Initial points with a zero gradient component lie in a finite union of hyperplanes in $\R^n$ and have Lebesgue measure zero. All measure and openness statements below concern initial points in $\R^n$.
\end{remark}

\subsection{The rolling method}\label{sec:RCR}
Choose integers $m\ge1$, $\ell\ge1$, and initialize by $m$ actual updates with one step $\tau>0$ chosen in advance. Assume $\tau\lambda_i\ne1$. The initialization polynomial is
\begin{equation}\label{eq:initialPolynomial}
 d(\xi)=(1-\tau\xi)^m,\qquad d(\lambda_i)\ne0\quad(i=1,\ldots,n).
\end{equation}
At times $t_k=m+km\ell$, extract the $m$ Ritz values from \eqref{eq:history}. Freeze them and use every reciprocal exactly $\ell$ times in $m\ell$ updates, with the last $m$ updates containing every reciprocal exactly once. The previous $m(\ell-1)$ updates can appear in any order. Repeating a complete sweep $\ell$ times is one allowed ordering, but is not required. Then refresh from the last $m$ consecutive gradients. All rolling results below hold for every ordering allowed by this rule, including permutations of the final sweep. The complete sweep at the end of the cycle is important: it makes the next block of $m$ stored gradients a Krylov basis generated by the gradient just before that sweep. The method is a member of the LMSDR family of \citet{ZouMagoules}.

Write $h_k=(h_{i,k})_{i=1}^n$. The oldest stored gradient $h_k$ defines the cycle-$k$ Ritz polynomial; $G_k$ receives its first Ritz step:
\begin{equation}\label{eq:anchors}
 h_k=g_{t_k-m}=g_{km\ell},\qquad G_k=g_{t_k},
 \qquad u_{i,k}=h_{i,k}^2.
\end{equation}
Here $h_0=g_0$ precedes the initialization, whereas $G_0=g_m$ follows it. The update index $t$ includes these $m$ initial steps, whose cost does not change the asymptotic rate. Let $\theta_{j,k}$, $j=1,\ldots,m$, be the Ritz values extracted from $H_{t_k}$. We write the cycle-$k$ polynomials as
\[
 \pi_k(\xi)=\det(\xi I_m-B_{t_k}),\qquad
 p_k(\xi)=\frac{\pi_k(\xi)}{\pi_k(0)}
          =\prod_{j=1}^m\left(1-\frac{\xi}{\theta_{j,k}}\right).
\]
By Lemma~\ref{lem:history}, $\pi_k$ is the monic polynomial satisfying \eqref{eq:orthogonal} for the historical weights $u_{i,k}$. The polynomial $p_k$ is therefore determined by $h_k$, but its steps act first on $G_k$. The previous complete sweep takes $h_k$ to $G_k$. At initialization, this link uses $d$ from \eqref{eq:initialPolynomial}. Eliminating $G_k$ then gives the recurrence for $h_k$:
\begin{align}
 G_0&=d(A)h_0,& G_k&=p_{k-1}(A)h_k\quad(k\ge1),\label{eq:G}\\
 h_1&=d(A)p_0(A)^{\ell-1}h_0,&
 h_{k+1}&=p_k(A)^{\ell-1}p_{k-1}(A)h_k\quad(k\ge1).
 \label{eq:B}
\end{align}
There are $(k+1)m\ell-(m+km\ell)=m(\ell-1)$ updates from $G_k$ to $h_{k+1}$. They use each new step $\ell-1$ times, giving the factor $p_k(A)^{\ell-1}$. The factor $p_{k-1}(A)$ comes from the final complete sweep of the previous cycle, from $h_k$ to $G_k$. All update factors are polynomials in $A$, so permuting the prefix does not change its endpoint. Permuting the final sweep changes the historical basis but not its Krylov span, by Lemma~\ref{lem:history}, and therefore preserves the Ritz values. For the initial points in our rolling convergence theorems, each stored gradient $h_k$ keeps all $n>m$ active distinct eigenvalues, so Lemma~\ref{lem:krylovRank} gives rank $m$ for every allowed permutation. Thus these permutations preserve the oldest stored gradients, cycle endpoints, and all later Ritz sets. For $m=2$, $\ell=3$, label the steps by $1,2$ and use a vertical bar to separate the final sweep from the earlier updates: $1122\mid12$ is allowed; $1112\mid22$ is not covered by this argument. This indexing includes $\ell=1$: then $h_{k+1}=G_k$.

Figure~\ref{fig:historyTimeline}(a) shows the ordering for $k\ge1$. At initialization, the first link is $G_0=d(A)h_0$ rather than a Ritz sweep.

For $m=1$, \eqref{eq:small} gives $\theta=g_{t-1}^TAg_{t-1}/\|g_{t-1}\|^2$. With $s=x_t-x_{t-1}=-\alpha_{t-1}g_{t-1}$ and $y=g_t-g_{t-1}=-\alpha_{t-1}Ag_{t-1}$, the common factor $\alpha_{t-1}^2$ cancels in $s^Ts/(s^Ty)$, giving this reciprocal. Thus $m=\ell=1$ is BB1, and $m=1$, $\ell>1$ is the cyclic reuse of this historical BB1 step size~\citep{CBB}. For $\ell=1$, each cycle is one complete sweep and the next extraction uses its $m$ stored gradients. With the same initialization and ordering, this is the full-rank, uninterrupted LMSD rule. These are special choices of the parameters in the LMSDR version above.

\subsection{The sampled method}\label{sec:SCR}
In the sampled method, the history is rebuilt at the start of every cycle. The extra updates make the first gradient of the stored block the current cycle-start gradient. This changes the dependence of the Ritz polynomial, even though its extraction is unchanged. Here $F_k=g_{km(\ell+1)}$ is the cycle-start gradient, distinct from the oldest stored gradient in RCR, $h_k$. Take $m$ positive sampling steps chosen in advance, $\tau_{k,0},\ldots,\tau_{k,m-1}$. Use the historical gradients generated by these updates to extract all $m$ Ritz values. Let $p_k$ be the normalized Ritz polynomial \eqref{eq:filter} extracted at time $km(\ell+1)+m$. Although extraction follows the $m$ sampling updates, the oldest gradient in the stored block is still $F_k$. By Lemma~\ref{lem:history} the block spans $\K_m(A,F_k)$, so the monic polynomial satisfies \eqref{eq:orthogonal} with weights $(F_k)_i^2$. From the sampling endpoint, use each reciprocal $\ell$ times in any order, then begin the next sampling block. If
\begin{equation}\label{eq:dk}
 d_k(t)=\prod_{j=0}^{m-1}(1-\tau_{k,j}t),
\end{equation}
then its exact endpoint relation is
\begin{equation}\label{eq:A}
 F_{k+1}=d_k(A)p_k(A)^\ell F_k,
\end{equation}
In Figure~\ref{fig:historyTimeline}(b), the factor $d_k(A)$ accounts for the change in the gradient during sampling. Although $F_k$ defines $p_k$, the first Ritz step is applied to $d_k(A)F_k$.

For $m=1$ and cycle-start gradient $F$, the sampling step gives $s=-\tau F$, $y=-\tau AF$, and $\alpha_{\rm BB1}=\frac{s^Ts}{s^Ty}=\frac{F^TF}{F^TAF}$. The extracted BB1 step is defined by $F$, whereas rolling cyclic BB1 uses the oldest stored gradient before the cycle start. The endpoint is
\begin{equation}\label{eq:scalarSamplingInterpretation}
 F_+=(I-\alpha_{\rm BB1}A)^\ell(I-\tau A)F
     =(I-\tau A)(I-\alpha_{\rm BB1}A)^\ell F.
\end{equation}

\begin{figure}[!htbp]
\centering
\begin{tikzpicture}[x=1cm,y=1cm,font=\small,>=stealth]
 \node[anchor=west,font=\small\bfseries] at (-0.5,0.9) {(a) Rolling history};
 \node (r0) at (0,0) {$h_k$};
 \node (r1) at (3.2,0) {$G_k$};
 \node (r2) at (6.4,0) {$h_{k+1}$};
 \node (r3) at (9.6,0) {$G_{k+1}$};
 \draw[->] (r0) -- node[above] {$p_{k-1}(A)$} (r1);
 \draw[->] (r1) -- node[above] {$p_k(A)^{\ell-1}$} (r2);
 \draw[->] (r2) -- node[above] {$p_k(A)$} (r3);
 \node[below=6pt,font=\scriptsize,align=center] at (r0.south) {defines $p_k$};
 \node[below=6pt,font=\scriptsize,align=center] at (r1.south) {cycle starts;\\new Ritz steps begin};
 \node[below=6pt,font=\scriptsize,align=center] at (r2.south) {defines $p_{k+1}$};
 \node[below=6pt,font=\scriptsize,align=center] at (r3.south) {next cycle starts};
 \node[anchor=west,font=\small\bfseries] at (-0.5,-1.4) {(b) Rebuilt history};
 \node (s0) at (0,-2.3) {$F_k$};
 \node (s1) at (4.8,-2.3) {$d_k(A)F_k$};
 \node (s2) at (9.6,-2.3) {$F_{k+1}$};
 \draw[->] (s0) -- node[above] {$m$ sampling updates} (s1);
 \draw[->] (s1) -- node[above] {$p_k(A)^\ell$} (s2);
 \node[below=6pt,font=\scriptsize,align=center] at (s0.south) {cycle starts;\\defines $p_k$};
 \node[below=6pt,font=\scriptsize,align=center] at (s1.south) {Ritz extraction;\\new Ritz steps begin};
 \node[below=6pt,font=\scriptsize,align=center] at (s2.south) {next cycle starts};
\end{tikzpicture}
\caption{Gradients defining and receiving the Ritz steps. In (a), the history begins at $h_k$, before the cycle-start gradient $G_k$. In (b), the sampling history begins at $F_k$, so its weights determine $p_k$; the extraction takes place after the sampling updates. Arrows represent blocks of actual updates, with the polynomial factors indicated for the Ritz blocks.}
\label{fig:historyTimeline}

\end{figure}

Both methods compute Ritz values from a block of gradients. RCR uses the last updates of the previous cycle, while SCR rebuilds the block at the current cycle start. This is the difference between the two histories and the reason for using the distinct symbols $h_k$ and $F_k$.

\subsection{Convergence rates}\label{sec:rates}

A nonnegative error sequence $e_t\to0$ converges R-superlinearly if $e_t\le\bar e_t$ eventually for some positive sequence $\bar e_t$ with $\bar e_{t+1}/\bar e_t\to0$; equivalently, $\limsup_{t\to\infty}e_t^{1/t}=0$~\citep[Chapter~9]{OrtegaRheinboldt}. Q-superlinear convergence of the error itself means $e_{t+1}/e_t\to0$ when the errors are positive. For a positive error $e_t\to0$, its R-order is $\liminf_{t\to\infty}(-\log e_t)^{1/t}$. Thus $e_t\le C\exp(-ca^t)$ gives R-order at least $a$, with equality under a matching lower envelope.

Since $\lambda_1\|x_t\|\le\|g_t\|\le\lambda_n\|x_t\|$, the gradient norm $\|g_t\|$ and the distance $\|x_t\|$ to the minimizer have the same R-superlinear convergence properties. We first estimate gradients at fixed intervals. The next lemma shows how bounds at these selected iterations give the rate at every intervening update.

\begin{lemma}[From Ritz blocks to individual updates]\label{lem:transfer}
Set $M=\max\{1,\lambda_n/\lambda_1-1\}$. Every Ritz update then satisfies $\|g_{t+1}\|\le M\|g_t\|$. Let $L\ge1$ be fixed and $a_k=g_{t_0+kL}$. Suppose each interval between consecutive $a_k$ consists of $L$ Ritz updates. Then every intervening gradient satisfies
\begin{equation}\label{eq:bracket}
 M^{-L}\|a_{k+1}\|\le\|g_t\|\le M^L\|a_k\|.
\end{equation}
To obtain the rate for all updates, suppose $C\ge c>0$, $\sigma>1$, and $\exp(-C\sigma^k)\le\|a_k\|\le\exp(-c\sigma^k)$ for all sufficiently large $k$. Then
\begin{equation}\label{eq:rate}
 \lim_{t\to\infty}\frac{\log(-\log\|g_t\|)}t
 =\frac{\log\sigma}{L}.
\end{equation}
\end{lemma}
\begin{proof}\hspace{0.5em}
For a Ritz value $\theta_j\in[\lambda_1,\lambda_n]$, $0<\lambda_i/\theta_j\le\lambda_n/\lambda_1$, so $|1-\lambda_i/\theta_j|\le\max\{1,\lambda_n/\lambda_1-1\}=M$. We write $t=t_0+kL+j$, with $0\le j\le L$, and apply the update bound step by step to get
\[
 \|g_t\|\le M^j\|a_k\|\le M^L\|a_k\|,
 \qquad
 \|a_{k+1}\|\le M^{L-j}\|g_t\|\le M^L\|g_t\|.
\]
The second inequality gives the lower bound in \eqref{eq:bracket} without inverting any update. Under the assumed bounds, $-\log\|g_t\|$ lies between $c\sigma^k-L\log M$ and $C\sigma^{k+1}+L\log M$ for large $k$. Thus $\log(-\log\|g_t\|)=k\log\sigma+O(1)$. Since $t=kL+O(1)$, \eqref{eq:rate} follows.
\end{proof}

\section{An exact convergence rate for \texorpdfstring{$n=m+1$}{n=m+1}}\label{sec:codimone}

When $n\le m$ and the history spans all $n$ directions, one complete Ritz sweep terminates, as described in Remark~\ref{rem:rank}. Thus $n=m+1$ is the first dimension beyond finite termination. The case $m=\ell=1$, $n=2$, is the classical two-dimensional BB1 setting~\citep{BB,Dai2013}. We now let $m$ be arbitrary and keep $n=m+1$. In this dimension, the equations \eqref{eq:orthogonal} for the values $u_{i,k}\pi_k(\lambda_i)$ have an $m\times(m+1)$ Vandermonde matrix. Its null space has dimension one, so finding one nonzero vector in it determines the ratios of all these values. We use Lagrange interpolation to find that vector. We then put the polynomial values into the recurrence for the history and obtain a second-order linear recurrence.

Throughout this section $n=m+1$ and $u_{i,0}=g_{i,0}^2>0$ for every $i$. Thus all $n$ spectral components are present initially. The interpolation weights $a_i$, Lagrange basis polynomials $L_i$, and positive interpolation constants $b_i$ are
\begin{equation}\label{eq:barycentric}
 a_i=\frac1{\prod_{j\ne i}(\lambda_i-\lambda_j)},\qquad
 L_i(t)=\prod_{j\ne i}\frac{t-\lambda_j}{\lambda_i-\lambda_j},
 \qquad b_i=|L_i(0)|>0.
\end{equation}
Here $L_i$ is the degree-$(n-1)=m$ polynomial satisfying $L_i(\lambda_j)=1$ for $j=i$ and $L_i(\lambda_j)=0$ for $j\ne i$. The number $a_i$ is its leading coefficient, and $b_i$ is the absolute value of $L_i$ at zero. For the historical weights $u_{i,k}=h_{i,k}^2$, fix a reference index $r\in\{1,\ldots,n\}$. The calculation in Lemma~\ref{lem:exact} below will show that
\[
 \left|\frac{p_k(\lambda_i)}{p_k(\lambda_r)}\right|
 =\left|\frac{a_i}{a_r}\right|\frac{u_{r,k}}{u_{i,k}}.
\]
We subtract $\log|a_i/a_r|$ from the log weight ratio so that the logarithm of this polynomial-value ratio is $-w_{i,k}$. Define
\begin{equation}\label{eq:w}
 w_{i,k}=\log\frac{u_{i,k}}{u_{r,k}}-\log\left|\frac{a_i}{a_r}\right|,
 \qquad w_{r,k}=0.
\end{equation}
Write $w_k=(w_{i,k})_{i\ne r}\in\R^{n-1}$, and set $w_{r,k}=0$ when the sum includes the reference index. Intuitively, a larger $w_{i,k}$ means that $|p_k(\lambda_i)|$ is smaller relative to $|p_k(\lambda_r)|$, so component $i$ is reduced more strongly relative to the reference component in this cycle.

Having defined $w_{i,k}$, the next step is to express $p_k(\lambda_i)$ explicitly in terms of the log ratios and determine how $w_k$ changes with $k$. The following lemma shows that $|p_k(\lambda_i)|$ has an explicit expression depending only on the log ratios $w_{j,k}$ and the constants $b_j$, while $w_k$ satisfies a second-order linear recurrence. This recurrence is the starting point for all the rate analysis that follows.

\begin{lemma}[Exact logarithmic dynamics]\label{lem:exact}
For \algRCR\ with $n=m+1$, full-support initial data, and a constant initial step size $\tau>0$ satisfying $1-\tau\lambda_i\ne0$ for every $i$,
\begin{align}
 p_k(\lambda_i)&=\frac{(-1)^{i-1}e^{-w_{i,k}}}{B(w_k)},
 &B(w)&=\sum_j b_j e^{-w_j},\label{eq:plicit}\\
 w_{k+1}&=(3-2\ell)w_k-2w_{k-1}.\label{eq:wrec}
\end{align}
The vector recurrence \eqref{eq:wrec} holds componentwise. Its first update is included by setting
\begin{equation}\label{eq:virtual}
 w_{i,-1}=-\log\left|\frac{d(\lambda_i)}{d(\lambda_r)}\right|.
\end{equation}
The index $-1$ is a virtual state encoding the prescribed initialization; no gradient at a negative update index is required. For every $k$ and $i$, $p_k(\lambda_i)\ne0$.
\end{lemma}
\begin{proof}\hspace{0.5em}
Comparing the leading coefficients in the Lagrange interpolation formula for $\xi^j$ gives
\[
 \sum_i a_i\lambda_i^j=0\quad(0\le j<m),\qquad
 \sum_i a_i\lambda_i^m=1.
\]
The first identities follow by interpolating $t^j$ with the polynomials $L_i$ for $j<m$: the coefficient of $t^{n-1}$ on the right is $\sum_i a_i\lambda_i^j$, whereas $t^j$ on the left has no $t^{n-1}$ term, so the sum is zero. For the last identity, $t^m$ is itself a monic degree-$m$ polynomial and its interpolant is the same polynomial. Comparing leading coefficients therefore gives one. The $m\times(m+1)$ Vandermonde matrix has rank $m$, so every vector $(c_i)$ satisfying $\sum_i c_i\lambda_i^j=0$ for $0\le j<m$ is proportional to $(a_i)$. By \eqref{eq:orthogonal}, the vector $(u_{i,k}\pi_k(\lambda_i))_{i=1}^n$ has zero inner product with every $(\lambda_i^j)_{i=1}^n$, $0\le j<m$. It therefore belongs to this null space, so there is a scalar $c$ such that $u_{i,k}\pi_k(\lambda_i)=c a_i$. Since $\pi_k$ is monic, comparison of the leading coefficients in its Lagrange interpolation formula gives $\sum_i a_i\pi_k(\lambda_i)=1$. Hence $1=\sum_i a_i\pi_k(\lambda_i)=c\sum_i a_i^2/u_{i,k}$, and therefore
\[
 \pi_k(\lambda_i)=\frac{a_i}{u_{i,k}\sum_j a_j^2/u_{j,k}}.
\]
Interpolation at zero gives $\pi_k(0)=\sum_i L_i(0)\pi_k(\lambda_i)$. The denominator defining $a_i$ has $n-i$ negative factors, whereas the numerator defining $L_i(0)$ has $n-1$. Hence
\[
 \operatorname{sgn}a_i=(-1)^{n-i},\qquad
 \operatorname{sgn}L_i(0)=(-1)^{i-1},\qquad
 \operatorname{sgn}(L_i(0)a_i)=(-1)^{n-1}.
\]
The same-sign denominator gives
\[
 |p_k(\lambda_i)|
 =\frac{|a_i|/u_{i,k}}{\sum_j b_j|a_j|/u_{j,k}}
 =\frac{e^{-w_{i,k}}}{B(w_k)},
\]
Here $|a_i|/u_{i,k}=(|a_r|/u_{r,k})e^{-w_{i,k}}$, so the common factor $|a_r|/u_{r,k}$ cancels from the numerator and denominator. Together with the signs above, this proves \eqref{eq:plicit}. In \eqref{eq:plicit}, each numerator is nonzero and $B(w_k)>0$. At the first transition, $d(\lambda_i)$ and $p_0(\lambda_i)$ are nonzero, so $h_1$ has full support. After this, \eqref{eq:B} multiplies every active component by nonzero values of the current and previous Ritz polynomials. Lemma~\ref{lem:krylovRank} therefore applies at each inductive stage, so every later Ritz extraction is well defined.

For $k\ge1$, the historical recurrence gives
\[
 u_{i,k+1}=|p_k(\lambda_i)|^{2(\ell-1)}
          |p_{k-1}(\lambda_i)|^2u_{i,k}.
\]
Since $\log|p_k(\lambda_i)/p_k(\lambda_r)|=-w_{i,k}$, taking ratios with $u_{r,k+1}$ gives $w_{i,k+1}=w_{i,k}-2(\ell-1)w_{i,k}-2w_{i,k-1}$. For the first transition, $h_1=d(A)p_0(A)^{\ell-1}h_0$ gives
\[
 w_{i,1}=(3-2\ell)w_{i,0}
          +2\log|d(\lambda_i)/d(\lambda_r)|
        =(3-2\ell)w_{i,0}-2w_{i,-1}.
\]
This proves \eqref{eq:wrec} from $k=0$ onward.
\end{proof}

Lemma~\ref{lem:exact} reduces the problem to the second-order linear recurrence \eqref{eq:wrec}. We next study how its solution grows with $k$. We first calculate the characteristic roots of the recurrence matrix, then ask whether the norm of the two-step state $(w_k,w_{k-1})$ grows like the $k$th power of the largest root modulus. We consider the two-step state rather than $w_k$ alone because the recurrence is of second order: $w_{k+1}$ depends on both $w_k$ and $w_{k-1}$.

The characteristic equation of each scalar recurrence in \eqref{eq:wrec} is
\begin{equation}\label{eq:characteristic}
 z^2-(3-2\ell)z+2=0.
\end{equation}
The largest modulus of its roots is
\begin{equation}\label{eq:sigma}
 \sig=\begin{cases}
 \sqrt2,&\ell=1,2,\\
 2,&\ell=3,\\
 \dfrac{2\ell-3+\sqrt{(2\ell-3)^2-8}}2,&\ell\ge4.
 \end{cases}
\end{equation}
For $\ell\ge3$, the characteristic root with smaller modulus is $\mu_s=-2/\sig$; thus $\mu_s=-1$ for $\ell=3$.

For $\ell\ge3$, the following lemma shows that the norm of the two-step state $(w_k,w_{k-1})$ has the same order of growth as $\sig^k$, unless the initial weights lie exactly in the exceptional direction determined by the root with smaller modulus. The exceptional set is explicit. For $\ell=1,2$, it is empty except for the special two-dimensional case described after Theorem~\ref{thm:B1}.

\begin{lemma}[Exponential growth of the two-step state]\label{lem:stateGrowth}
Under the assumptions of Lemma~\ref{lem:exact}, let $\sigma_\ell$ be given by \eqref{eq:sigma}. For $\ell=1,2$, exclude only the initial two-step state $(w_0,w_{-1})=(0,0)$. For $\ell\ge3$, fix any reference index $r$ and exclude the initial weights
\begin{equation}\label{eq:exception}
 \frac{u_{i,0}}{u_{r,0}}
 =\left|\frac{a_i}{a_r}\right|
  \left|\frac{d(\lambda_i)}{d(\lambda_r)}\right|^{-\mu_s}
 \quad(i\ne r).
\end{equation}
Equivalently, these weights are proportional to $|a_i|\,|d(\lambda_i)|^{-\mu_s}$, so the exception is independent of $r$. For every other full-support initial point, there are $c_0,C_0>0$ such that
\begin{equation}\label{eq:stategrowth}
 c_0\sigma_\ell^k\le\|(w_k,w_{k-1})\|\le C_0\sigma_\ell^k
\end{equation}
for all sufficiently large $k$. The constants depend on the initial point and the problem and algorithm parameters.
\end{lemma}
\begin{proof}\hspace{0.5em}
The recurrence matrix for $(w_k,w_{k-1})$ is
\[
 C_\ell=\begin{pmatrix}3-2\ell&-2\\1&0\end{pmatrix}.
\]
For $\ell=1,2$, let $\alpha+i\beta$ be one of its nonreal characteristic roots. Then $\alpha^2+\beta^2=2$ and $\beta\ne0$. Set
\[
 S=\begin{pmatrix}\alpha&\beta\\1&0\end{pmatrix},
 \qquad
 S^{-1}C_\ell S=\begin{pmatrix}\alpha&\beta\\-\beta&\alpha\end{pmatrix};
\]
the last matrix is $\sqrt2$ times an orthogonal matrix. Apply this change of basis to each coordinate pair. Its powers multiply the transformed state norm by $(\sqrt2)^k$, and the matrices $S,S^{-1}$, which are independent of $k$, give, for every nonzero initial pair,
\[
 c_0\sig^k\le\|(w_k,w_{k-1})\|\le C_0\sig^k.
\]
For $\ell\ge3$, the roots are $-\sig$ and $\mu_s$. The coefficient of $(-\sig)^{k+1}$ is the vector $(w_0-\mu_s w_{-1})/(-\sig-\mu_s)$. It vanishes exactly when $w_0=\mu_s w_{-1}$ in every coordinate, which is \eqref{eq:exception}. Otherwise this coefficient vector is nonzero and \eqref{eq:stategrowth} holds for large $k$. This bound concerns the two-step state; an individual coordinate can be small at some iterations.
\end{proof}

The bound in Lemma~\ref{lem:stateGrowth} is for the norm of the full two-step state $(w_k,w_{k-1})$, not a separate lower bound for each coordinate $w_{i,k}$; an individual coordinate can be small at some iterations. Our final goal, however, is the gradient norm. We need the cumulative decrease of every spectral component, not the size of $w$ itself. The next step is therefore to turn the growth of the full state into a bound on how much each component decreases over the first $N$ cycles. For this purpose, define the nonnegative quantities
\[
 v_{i,k}=\log\sum_j e^{w_{i,k}-w_{j,k}}\ge0,
 \qquad V_{i,N}=\sum_{k=0}^{N-1}v_{i,k}.
\]
These quantities measure cumulative component decrease: equation~\eqref{eq:plicit} gives
\begin{equation}\label{eq:lossidentity}
 \log|p_k(\lambda_i)|=-v_{i,k}-\varepsilon_k,
 \quad \log\min_jb_j\le\varepsilon_k\le\log\max_jb_j.
\end{equation}
Here $v_{i,k}\ge0$ because the term with $j=i$ in the sum is one. The quantity $-\log|p_k(\lambda_i)|$ measures the logarithmic decrease of component $i$ caused by the cycle's Ritz polynomial; it equals $v_{i,k}$ plus a uniformly bounded correction $\varepsilon_k$. The following lemma shows that, whenever the two-step state grows like $\sigma^k$, every $V_{i,N}$ grows like $\sigma^N$, with upper and lower bounds of the same order.

\begin{lemma}[From ratio growth to cumulative component decrease]\label{lem:cumulativeLoss}
Let $\ell\ge1$ and let $w_k$ satisfy \eqref{eq:wrec} for $k\ge0$, with a given virtual state $w_{-1}$. Suppose that for some $\sigma>1$ and $c_0,C_0>0$,
\[
 c_0\sigma^k\le\|(w_k,w_{k-1})\|\le C_0\sigma^k
\]
for all sufficiently large $k$. Then, for every $i$, there are $c_i,C_i>0$ such that $c_i\sigma^N\le V_{i,N}\le C_i\sigma^N$ for all sufficiently large $N$. The constants depend on the initial state, $\ell$, $n$, and the constants in the assumed bound.
\end{lemma}
\begin{proof}\hspace{0.5em}
On changing the reference from $r$ to $s$, the coordinates become $w_i^{(s)}=w_i^{(r)}-w_s^{(r)}$, with inverse $w_i^{(r)}=w_i^{(s)}-w_r^{(s)}$. The same formulas hold at the previous time and for the virtual initial state. Thus this is an invertible linear transformation of the two-step state, independent of $k$, and the assumed two-step growth bound holds for each reference, with constants depending on the reference.

We now fix a reference $r$ and obtain both an upper and a lower bound for every $w_{i,k}$ in terms of $V_{r,N}$. Since the sum defining $v_{r,k}$ contains $e^{-w_{i,k}}$, we have $w_{i,k}\ge-v_{r,k}\ge-V_{r,N}$ for $k<N$. Also $v_{r,k}-v_{i,k}=-w_{i,k}$, so
\[
 w_{i,k+1}-w_{i,k}
 =2(\ell-1)(v_{r,k}-v_{i,k})
   +2(v_{r,k-1}-v_{i,k-1})\quad(k\ge1).
\]
For the upper bound, we drop only $-2(\ell-1)v_{i,k}$ and $-2v_{i,k-1}$, which are nonpositive. The initialized first step gives $w_{i,1}-w_{i,0}=2(\ell-1)(v_{r,0}-v_{i,0})-2w_{i,-1}$. We sum the later differences and drop the nonpositive terms to obtain, for $k\ge1$,
\[
 w_{i,k}\le w_{i,0}-2w_{i,-1}
              +2(\ell-1)V_{r,k}+2V_{r,k-1}.
\]
The lower bound and the summed recurrence now control the whole state using only the decrease of the reference component. Absorbing the initial terms into $C_r$ gives $-V_{r,N}\le w_{i,k}\le C_r+2\ell V_{r,N}$ for $0\le k<N$. Consequently, for large $N$,
\[
 c_r'\sigma^{N-1}
 \le\|(w_{N-1},w_{N-2})\|
 \le C_1(1+V_{r,N}),
\]
and hence $V_{r,N}\ge c_r\sigma^N-C_r$. Since $r$ was arbitrary, choosing each component as reference gives this lower bound for every $V_{i,N}$; the constants depend on the reference, but $\sigma$ does not. For the reverse bound, $v_{i,k}\le\log n+2\|w_k\|_\infty\le\log n+C\sigma^k$. Summing $\sum_{k<N}\sigma^k=(\sigma^N-1)/(\sigma-1)$ gives, for every $i$,
\begin{equation}\label{eq:V}
 c_i\sigma^N\le V_{i,N}\le C_i\sigma^N.
\end{equation}
\end{proof}

We have now proved that the cumulative decrease of every component grows on the scale $\sig^N$. Two steps remain. First, we pass from the cumulative quantities $V_{i,N}$ back to the historical gradients $h_k$ to obtain doubly exponential upper and lower bounds for $\|h_k\|$. Then we use Lemma~\ref{lem:transfer} to transfer the rate at cycle endpoints to every gradient update within a cycle.

\Needspace{5\baselineskip}
\begin{theorem}[Exact rate and exceptional initial weights]\label{thm:B1}
Let $m\ge1$, $\ell\ge1$, and $n=m+1$. Apply \algRCR\ with $m$ initial updates with the same step size $\tau>0$ such that $1-\tau\lambda_i\ne0$ for every active eigenvalue, as in Section~\ref{sec:RCR}. Let $\sigma_\ell$ be given by \eqref{eq:sigma}. Exclude the initial states specified in Lemma~\ref{lem:stateGrowth}, including $(w_0,w_{-1})=(0,0)$ when $\ell=1,2$. Every other full-support initial point satisfies, for sufficiently large $k$,
\begin{equation}\label{eq:B1envelope}
 e^{-C\sig^k}\le\|h_k\|\le e^{-c\sig^k},\qquad C\ge c>0,
\end{equation}
where $C,c$ depend on the initial point and the problem and algorithm parameters. The gradient sequence, including every update within a cycle, satisfies
\begin{equation}\label{eq:B1rate}
 \lim_{t\to\infty}\frac{\log(-\log\|g_t\|)}t
 =\frac{\log\sig}{m\ell}.
\end{equation}
The same limit holds with $\|x_t\|$ in place of $\|g_t\|$. The exceptional initial points form a Lebesgue-null set in $\R^n$.
\end{theorem}
\begin{proof}\hspace{0.5em}
Lemma~\ref{lem:stateGrowth} gives \eqref{eq:stategrowth}. Lemma~\ref{lem:cumulativeLoss}, applied with $\sigma=\sigma_\ell$, gives \eqref{eq:V}. We have proved that the cumulative decrease has the same exponential scale for every component. We now return from the ratios to the original historical gradients. The exact product of the historical Ritz polynomials is
\[
 h_N=d(A)\left[\prod_{k=0}^{N-1}p_k(A)^{\ell-1}\right]
            \left[\prod_{k=0}^{N-2}p_k(A)\right]h_0
 \qquad(N\ge1).
\]
Using \eqref{eq:lossidentity} in each spectral component gives $\log|h_{i,N}|=C_i'-(\ell-1)V_{i,N}-V_{i,N-1}+O(N)$, where $C_i'=\log|d(\lambda_i)h_{i,0}|$. The $O(N)$ term comes from summing the uniformly bounded corrections $\varepsilon_k$ in \eqref{eq:lossidentity}. For $\ell=1$, the term $-V_{i,N-1}$ alone gives the required decrease. Equation~\eqref{eq:V} proves \eqref{eq:B1envelope}. Lemma~\ref{lem:transfer} proves \eqref{eq:B1rate}. The norm equivalence between $g_t$ and $x_t$ gives the result for the iterate error.

The exceptional positive weight ratios are absent or a single point in $\R^{n-1}$. Since $u_{i,0}=g_{i,0}^2$, the corresponding nonzero initial gradients lie on finitely many rays in $\R^n$, a null set because $n=m+1\ge2$. The invertible map $g_0=Ax_0$ preserves null sets, proving the claim for initial points.
\end{proof}

For $\ell=1,2$, $w_{-1}=0$ requires all $|1-\tau\lambda_i|$ to be equal, allowing at most two distinct eigenvalues. The exceptional state therefore occurs precisely when $m=1$, $n=2$, $\tau=2/(\lambda_1+\lambda_2)$, and $u_{1,0}=u_{2,0}$; otherwise it is absent.

The R-order in Theorem~\ref{thm:B1} is $\sigma_\ell$ per oldest stored gradient and $\sigma_\ell^{1/(m\ell)}$ per gradient update. Since $\sigma_1=\sqrt2$, $m=1$, $\ell=1$ gives R-order $\sqrt2$ for BB1, whereas $m=2$, $\ell=1$, $n=3$ gives R-order $(\sqrt2)^{1/2}=2^{1/4}$ per update of a complete LMSD sweep.

\section{Convergence on open sets for \texorpdfstring{$m+1<n\le2m$}{m+1<n<=2m}}\label{sec:two}

We now extend the superlinear convergence analysis of the rolling method from $n=m+1$ to $m+1<n\le2m$. In Section~\ref{sec:codimone}, the homogeneous equations \eqref{eq:orthogonal} for $u_{i,k}\pi_k(\lambda_i)$ had a one-dimensional null space. This gave the polynomial values up to a common scalar factor. For $n>m+1$ the null space has dimension $n-m>1$, so that argument no longer gives the exact recurrence. We need a different way to estimate $p_k(\lambda_i)$.

The idea comes from BB in two dimensions. When one component is much larger than the other, the BB recurrence strongly reduces the larger component, and the two exchange roles. In a complete LMSD sweep, the $m$ reciprocal Ritz values give a polynomial of degree $m$. If its roots were exactly the eigenvalues of $m$ dominant components, it would eliminate all $m$ of them. This suggests replacing the two components by two groups of $m$ components and starting with $n=2m$. In the actual iteration the roots need only be close enough to reduce the dominant components sharply. When the $m$ dominant eigenvalues lie on one side of the rest of the spectrum, the one-sided polynomial estimate makes this reduction precise and bounds the polynomial values on the other group above and away from zero. Since the history recurrence involves both $p_k$ and $p_{k-1}$, we still need to show that the groups keep exchanging roles.

We do this by comparing how fast the gap between groups grows with how fast the differences within a group grow. A large enough gap allows us to keep using the polynomial estimate. The recurrence from that estimate then shows that the gap grows faster than the differences within the groups. The two groups are disjoint when $n=2m$. When $m+1<n<2m$, they overlap, and we also need to track the common components.

\subsection{Two disjoint groups when \texorpdfstring{$n=2m$}{n=2m}}\label{sec:separated}

Let $n=2m$, $m\ge2$, and divide the ordered spectrum into $I=\{1,\ldots,m\}$ and $J=\{m+1,\ldots,2m\}$. Write $P_{i,k}=\log|h_{i,k}|=\tfrac12\log u_{i,k}$ and let $\bar P_{K,k}=m^{-1}\sum_{i\in K}P_{i,k}\quad(K=I,J)$. Define the initial gap and spread by
\begin{equation}\label{eq:DS}
 D_0=\bar P_{I,0}-\bar P_{J,0},\qquad
 S_0=\max_{K\in\{I,J\}}\max_{i\in K}|P_{i,0}-\bar P_{K,0}|.
\end{equation}
Thus $D_0$ is the logarithm of the ratio of the geometric mean component sizes of the two groups, while $S_0$ measures differences within them. Equal absolute component values within each group give $S_0=0$. Multiplying all components in $I$ by the same positive factor makes $D_0$ arbitrarily large without changing $S_0$.

A large initial gap does not by itself ensure separation in later cycles. To check separation at each cycle, define
\[
 D_k=(-1)^k(\bar P_{I,k}-\bar P_{J,k}),\qquad
 S_k=\max_{K\in\{I,J\}}\max_{i\in K}|P_{i,k}-\bar P_{K,k}|.
\]
The sign $(-1)^k$ sets the order we want to prove. Let $K_k=I$ at even cycles and $K_k=J$ at odd cycles, and let $L_k$ be the other group. The mean gap is $D_k$, and the largest difference from a group mean is $S_k$. The gap that determines whether all components are in the desired order is
\[
 \Gamma_k=\min_{i\in K_k}P_{i,k}-\max_{j\in L_k}P_{j,k}
 \ge D_k-2S_k.
\]
This is why we track both $D_k$ and $S_k$. If $D_k$ exceeds $2S_k$ by enough, every component in $K_k$ is separated from every component in $L_k$.

We next use the polynomial estimate in the recurrence for $h_k$. Lemma~\ref{lem:twobound} gives constants $\Gamma_*>0$ and $C$ that depend only on the spectrum and $m$. If $\Gamma_k\ge\Gamma_*$, put $M_{L_k,k}=\max_{j\in L_k}P_{j,k}$. Then
\[
 \log|p_k(\lambda_i)|=
 \begin{cases}
  2(M_{L_k,k}-P_{i,k})+e_{i,k},&i\in K_k,\\
  e_{i,k},&i\in L_k,
 \end{cases}
 \qquad |e_{i,k}|\le C.
\]
Thus $|p_k(\lambda_i)|\le e^Ce^{-2\Gamma_k}$ on the dominant group, while $e^{-C}\le|p_k(\lambda_i)|\le e^C$ on the other group. On the latter group the change on the log scale is bounded. The estimate holds uniformly over the weight ratios within each group. Its proof is in Appendix~\ref{app:constants}.

Throughout this subsection, let $\ell\ge4$ and use the constant-step initialization of Section~\ref{sec:RCR}, with $1-\tau\lambda_i\ne0$ for every $i$. Substituting the polynomial values into $h_{k+1}=p_k(A)^{\ell-1}p_{k-1}(A)h_k$ and taking the difference of the two mean log amplitudes gives
\[
 D_{k+1}=aD_k-2D_{k-1}+E_k,\qquad a=2\ell-3.
\]
The error $E_k$ comes from the differences within the groups and the bounded polynomial errors. Write $R_0=S_0+1$ and $\nu=\sqrt{2\ell-1}$. As long as the groups have been separated at the preceding cycles, the calculations in Appendix~\ref{app:rollingRecurrences} give
\[
 S_k\le C_sR_0\nu^k,\qquad |E_k|\le C_eR_0\nu^k.
\]
The first transition is $D_1=aD_0+E_0$, with $|E_0|\le C_eR_0$. Setting the auxiliary value $D_{-1}=0$ puts this transition in the same form as the later ones. The homogeneous characteristic equation is $z^2-(2\ell-3)z+2=0$. Multiplying the solution of the recurrence in Section~\ref{sec:codimone} by $(-1)^k$ gives this homogeneous recurrence. The roots therefore differ only in sign, and their largest modulus is the same $\sigma=\sigma_\ell$ as in \eqref{eq:sigma}. For $\ell\ge4$, we have $\nu<\sigma$. The spread and the errors thus grow more slowly than $\sigma^k$. Their accumulated effect changes the coefficient of $\sigma^k$ by a bounded amount, so choosing $D_0$ large compared with $R_0$ keeps that coefficient positive.

The next lemma assumes separation for $0\le k<N$ and gives bounds through $k=N$. This choice of indices lets us check a possible first failure at $N$. The earlier separation gives the bounds at $N$, and a large enough $D_0$ makes those bounds imply separation there as well.

\begin{lemma}[Comparison of the gap and spread]\label{lem:rollingComparison}
Suppose the initial gradient has full support and $D_0\ge0$. Let $N\ge1$ and assume $\Gamma_k\ge\Gamma_*$ for $0\le k<N$. Put $R_0=S_0+1$. There are constants $C_s,C_f\ge1$ and $C_d>0$, depending only on the spectrum, $m,\ell$, and the initialization step, such that for $0\le k\le N$,
\begin{equation}\label{eq:Dcontrol}
 (D_0-C_dR_0)\sigma^k\le D_k
 \le(C_fD_0+C_dR_0)\sigma^k,\qquad
 S_k\le C_sR_0\nu^k.
\end{equation}
\end{lemma}
\begin{proof}\hspace{0.5em}
Lemma~\ref{lem:rollingRecurrences} gives the gap recurrence, and Lemma~\ref{lem:rollingSpread} gives the spread bound and $|E_k|\le C_eR_0\nu^k$, including initialization. It remains to bound the accumulated effect of the errors on the gap. Put $\sigma_-=2/\sigma$ and write
\[
 \varphi_k=\frac{\sigma^{k+1}-\sigma_-^{k+1}}{\sigma-\sigma_-},
 \qquad C_f=\frac{\sigma}{\sigma-\sigma_-}.
\]
With $D_{-1}=0$, the formula in Lemma~\ref{lem:forcedRecurrence} gives
\[
 D_k=D_0\varphi_k+\sum_{j=0}^{k-1}\varphi_{k-1-j}E_j,
 \qquad \sigma^k\le\varphi_k\le C_f\sigma^k.
\]
Since $\nu<\sigma$,
\[
 \left|\sum_{j=0}^{k-1}\varphi_{k-1-j}E_j\right|
 \le C_eC_fR_0\sum_{j=0}^{k-1}\sigma^{k-1-j}\nu^j
 \le \frac{C_eC_f}{\sigma-\nu}R_0\sigma^k.
\]
Take $C_d=C_eC_f/(\sigma-\nu)$. Combining this error bound with the lower and upper bounds for $D_0\varphi_k$ proves the two bounds for the gap. Lemma~\ref{lem:rollingSpread} already gives the spread bound.
\end{proof}

\Needspace{5\baselineskip}
\begin{theorem}[Superlinear convergence under initial separation]\label{thm:B2}
Let $n=2m$, $m\ge2$, and $\ell\ge4$. Initialize \algRCR\ by $m$ updates using the same step size $\tau>0$, where $1-\tau\lambda_i\ne0$ for every $i$. For each $s>0$, there is a finite $D_*(s)$ such that every full-support initial point $x_0\in\R^n$ with
\begin{equation}\label{eq:openset}
 S_0<s,\qquad D_0>D_*(s)
\end{equation}
generates full-rank histories with no active spectral component vanishing. The oldest stored gradients satisfy \eqref{eq:B1envelope} and have exact R-order $\sig$; the full gradient sequence satisfies \eqref{eq:B1rate} and has exact R-order $\sig^{1/(m\ell)}$ per update. Here $\sig$ is given by \eqref{eq:sigma}. The set \eqref{eq:openset} is nonempty and open in $\R^n$. The same conclusion holds for $D_0<-D_*(s)$, with $J$ initially dominant.
\end{theorem}
\begin{proof}\hspace{0.5em}
Put $R_0=S_0+1$ and take $C_s,C_f,C_d$ from Lemma~\ref{lem:rollingComparison}. We first establish separation at the initial cycle, then show that it cannot fail at a later cycle. Choose
\[
 D_0>\max\{2C_dR_0,8C_sR_0,4\Gamma_*\}.
\]
Since $C_s\ge1$ and $S_0<R_0$, this gives $2S_0<D_0/4$. Hence $\Gamma_0\ge D_0-2S_0>3D_0/4>\Gamma_*$, so the initial groups are separated.

Suppose the separation holds for $0\le k<N$. Lemma~\ref{lem:rollingComparison} then applies through cycle $N$. Since $\nu<\sigma$, for every $0\le k\le N$ it gives
\[
 D_k\ge(D_0-C_dR_0)\sigma^k\ge\tfrac12D_0\sigma^k,
 \qquad 2S_k\le2C_sR_0\nu^k<\tfrac14D_0\sigma^k.
\]
Subtracting the second bound from the first proves
\begin{equation}\label{eq:invariance}
 D_k\ge\tfrac12D_0\sigma^k,\qquad
 \Gamma_k\ge D_k-2S_k>\tfrac14D_0\sigma^k>\Gamma_*.
\end{equation}
If the alternation between separated groups first failed at $N$, this inequality would give the required ordering and separation at $N$, a contradiction. For $N=1$, the lemma includes the initialized transition, whose gap formula is \eqref{eq:Dinit}. In this case the bounds read
\[
 D_1\ge\tfrac12D_0\sigma,\qquad
 \Gamma_1\ge\tfrac12D_0\sigma-2C_sR_0\nu
              >\tfrac14D_0\sigma>\Gamma_*.
\]
Thus the argument also covers the first transition. A threshold depending only on $s$ follows by $R_0<s+1$. The polynomial bound keeps all components nonzero, including during each Ritz sweep. Full support and $n>m$ make the $m$ Krylov vectors linearly independent, so Lemma~\ref{lem:krylovRank} gives rank $m$ for every stored block.

We next use the separation to bound the decrease of every component.
For $i\in K_k$, $\log|p_k(\lambda_i)|$ is at most $-2\Gamma_k+C$, while $\log|p_{k-1}(\lambda_i)|$ is bounded. For $i\in L_k$, $\log|p_{k-1}(\lambda_i)|$ is at most $-2\Gamma_{k-1}+C$ and $\log|p_k(\lambda_i)|$ is bounded. Hence
\begin{align*}
 P_{i,k+1}-P_{i,k}&\le-2(\ell-1)\Gamma_k+\ell C&& (i\in K_k),\\
 P_{i,k+1}-P_{i,k}&\le-2\Gamma_{k-1}+\ell C&& (i\in L_k).
\end{align*}
Thus the total change in each $P_{i,k}$ is negative and of order $\sigma^k$. Conversely, for $i\in K_k$, $|\log|p_k(\lambda_i)||\le2(D_k+2S_k)+C$, while for $i\in L_k$ it is at most $C$. Equations~\eqref{eq:Sbound} and \eqref{eq:Dcontrol} therefore bound every $\log|p_k(\lambda_i)|$ by $C'\sigma^k$ in absolute value. The geometric sum of these bounds is $O(\sigma^N)$; the bounded logarithmic terms for components in $L_k$ contribute only $O(N)$. This proves the lower envelope and hence the two-sided bound. Finally, Lemma~\ref{lem:transfer} proves the rate for all actual updates.

Finally, we check that the initial conditions define a nonempty open set. On the open set of full-support gradients, each $P_{i,0}=\log|h_{i,0}|$ is continuous, so $S_0$ and $D_0$ are continuous as well. The strict inequalities \eqref{eq:openset} therefore define an open set there. Take equal values of $P_{i,0}$ within each group and make the difference between the two levels larger than $D_*(s)$. This gives a gradient in the set. Since $h_0=g_0=Ax_0$ and $A$ is invertible, the corresponding initial points form a nonempty open set in $\R^n$. Scaling the entire initial gradient by a positive factor preserves the conditions, so initial gradients in the set can have arbitrarily small norm.
\end{proof}

\subsection{Overlapping groups when \texorpdfstring{$m+1<n<2m$}{m+1<n<2m}}\label{sec:overlap}

The preceding subsection used two disjoint groups of $m$ components to cover the $n=2m$ eigenvalues. When $m+1<n<2m$, two candidate dominant groups of size $m$ must overlap. The common components belong to the dominant group in two consecutive cycles, so both $p_k$ and $p_{k-1}$ reduce them strongly. Their recurrence differs from that of the outer components, which move in and out of the dominant group. We keep the gap and spread from the disjoint case and add a quantity to check that the common components stay between the two outer groups.

Let $n=m+r$ with $2\le r<m$ and partition the ordered support into $L=\{1,\ldots,r\},\ C=\{r+1,\ldots,m\},\ R=\{m+1,\ldots,m+r\}$. These are consecutive blocks in increasing eigenvalue order; their log amplitudes, not their spectral locations, alternate in size. The algorithm continues to compute and use all $m$ historical Ritz values. The outer groups $L,R$ each have $r$ points; the common group $C$ has $m-r=2m-n$ points. We seek trajectories in which $L\cup C$ dominates at even cycles and $C\cup R$ dominates at odd cycles, with the common group between the two outer groups in log amplitude.

As in Section~\ref{sec:separated}, write $P_{i,k}=\log|h_{i,k}|$. For $K=L,R$, define $\bar P_{K,k}=r^{-1}\sum_{i\in K}P_{i,k}$ and $S_k=\max_{K=L,R}\max_{i\in K}|P_{i,k}-\bar P_{K,k}|$. Let $T_k$ be the mean of $L$ at even $k$ and of $R$ at odd $k$, and let $U_k$ be the other outer mean. Set $D_k=T_k-U_k$ and $Y_{i,k}=P_{i,k}-(T_k+U_k)/2$ for $i\in C$. In particular,
\[
\begin{gathered}
 D_0=\bar P_{L,0}-\bar P_{R,0},\quad
 S_0=\max_{K=L,R}\max_{i\in K}|P_{i,0}-\bar P_{K,0}|,\\
 Y_{i,0}=P_{i,0}-\tfrac12(\bar P_{L,0}+\bar P_{R,0})\quad(i\in C).
\end{gathered}
\]

The offset $Y_{i,k}$ lets us check whether a common component lies between the outer groups. On the log scale it is $D_k/2-Y_{i,k}$ below the larger outer mean and $D_k/2+Y_{i,k}$ above the smaller outer mean. After allowing for the spread within each outer group, the condition
\[
 \frac{D_k}{2}-S_k-\max_{i\in C}|Y_{i,k}|>\Gamma_*
\]
gives a gap greater than $\Gamma_*$ between the dominant outer group and the common group, and between the common group and the weaker outer group. Thus Lemma~\ref{lem:twobound} applies to $L\cup C$ with complement $R$ at even cycles and to $C\cup R$ with complement $L$ at odd cycles.

We next find the recurrence for $Y_{i,k}$ and compare it with that for $D_k$. Use $a=2\ell-3$, $\sigma=\sigma_\ell$, and $\nu=\sqrt{2\ell-1}$ as in the disjoint case. The outer gap and spread obey the same recurrences as before. For a common component, subtracting the update of the midpoint of the outer means from its log update gives
\[
 Y_{i,k+1}=-aY_{i,k}-2Y_{i,k-1}+F_{i,k},
 \qquad |F_{i,k}|\le C_eR_0\nu^k.
\]
Here $R_0=S_0+1$ and the auxiliary initial value is $Y_{i,-1}=0$. Lemmas~\ref{lem:commonOffset} and~\ref{lem:overlapDynamics} give the subtraction and the error bound. Multiplying $Y_{i,k}$ by $(-1)^k$ gives the same homogeneous recurrence as for $D_k$. The offset can thus grow at the same rate $\sigma^k$. To keep the common components between the outer groups, we therefore choose $D_0$ large compared with the initial offsets and spread. The next lemma bounds the coefficients of $\sigma^k$ in both estimates. We then use these bounds to rule out a first failure of separation.

\begin{lemma}[Bounds for the outer groups and common components]\label{lem:overlapRecurrences}
Let $m+1<n<2m$ and $\ell\ge4$. Use the constant-step initialization of Section~\ref{sec:RCR}, with $1-\tau\lambda_i\ne0$ for all $i$ and a full-support initial gradient, and let $N\ge1$. Suppose that for $0\le k<N$ the proposed dominant set, $L\cup C$ at even $k$ and $C\cup R$ at odd $k$, is separated from its complement by at least $\Gamma_*$. For $D_0\ge0$ and $R_0=S_0+1$, constants chosen uniformly for the two sets give, for $0\le k\le N$ and $i\in C$,
\[
 D_k\ge(D_0-C_dR_0)\sigma^k,\quad
 D_k\le(C_fD_0+C_dR_0)\sigma^k,\quad
 |Y_{i,k}|\le(C_f|Y_{i,0}|+C_dR_0)\sigma^k.
\]
The spread also satisfies $S_k\le C_sR_0\nu^k$. The constants depend only on the spectrum, $m,\ell$, and the initialization step, and not on $N$ or the initial gradient.
\end{lemma}
\begin{proof}\hspace{0.5em}
Lemma~\ref{lem:overlapDynamics} gives the spread bound and the two forced recurrences, including initialization. Both forcings are bounded by $C_eR_0\nu^k$. Multiplication by $(-1)^k$ converts the common-offset recurrence to the gap recurrence. Lemma~\ref{lem:forcedRecurrence}, applied to $D_k$ and to $(-1)^kY_{i,k}$, gives the stated estimates with $C_d=C_eC_f/(\sigma-\nu)$. The calculation uses the same $C_e$ for both recurrences, as verified in Lemma~\ref{lem:commonOffset}.
\end{proof}

\Needspace{5\baselineskip}
\begin{theorem}[Superlinear convergence with overlapping groups]\label{thm:overlap}
Let $m+1<n<2m$ and $\ell\ge4$. Initialize \algRCR\ by $m$ updates with the same step size $\tau>0$ such that $1-\tau\lambda_i\ne0$ for every $i$. For every $s>0$ there exists $D_*(s)<\infty$ such that the strict conditions on full-support initial points $x_0\in\R^n$,
\begin{equation}\label{eq:overlapInitial}
 S_0<s,\qquad \max_{i\in C}|Y_{i,0}|<s,\qquad D_0>D_*(s)
\end{equation}
define a nonempty open set in $\R^n$ for which all histories have full rank and no active spectral component vanishes. At even cycles, every component in $L\cup C$ has larger absolute value than every component in $R$; at odd cycles, the same holds for $C\cup R$ relative to $L$. The method converges doubly exponentially, with exact R-order $\sigma_\ell^{1/(m\ell)}$ per gradient update. The historical oldest stored gradients satisfy the two-sided envelope \eqref{eq:B1envelope}, with $\sigma_\ell$ from \eqref{eq:sigma}, and the full actual gradient sequence satisfies the exact rate \eqref{eq:B1rate}. The corresponding iterate errors have the same rate.
\end{theorem}
\begin{proof}\hspace{0.5em}
As in the disjoint case, we first show that separation persists and then use the accumulated decrease of each component to obtain the rate. Here we also need to check that the common components stay between the outer groups. Put $R_0=S_0+1$ and take the constants from Lemma~\ref{lem:overlapRecurrences}. The same $C_d$ bounds the errors in the gap and offset estimates. Choose $D_*(s)$ as in \eqref{eq:overlapThreshold}. Since $R_0<s+1$ and $\max_{i\in C}|Y_{i,0}|<s$, the assumption $D_0>D_*(s)$ gives
\[
 D_0>\max\{2C_dR_0,16C_sR_0,
 16[C_f\max_{i\in C}|Y_{i,0}|+C_dR_0],8\Gamma_*\}.
\]
At $k=0$ these inequalities put the common components between the outer groups, with both adjacent gaps greater than $\Gamma_*$.

Suppose separation holds for $0\le k<N$. The bounds in Lemma~\ref{lem:overlapRecurrences} then hold through $k=N$. Since $D_k\ge(D_0-C_dR_0)\sigma^k>\tfrac12D_0\sigma^k$ and $\nu<\sigma$, for $0\le k\le N$ we have
\[
 \frac{S_k}{D_k}\le\frac{2C_sR_0}{D_0}<\frac18,
 \qquad
 \frac{|Y_{i,k}|}{D_k}
 \le\frac{2(C_f|Y_{i,0}|+C_dR_0)}{D_0}<\frac18,
 \qquad \frac{D_k}{4}>\Gamma_*.
\]
Thus the gaps from the dominant outer group to the common group and from the common group to the other outer group are both at least
\[
 \frac{D_k}{2}-S_k-\max_{i\in C}|Y_{i,k}|
 >\frac{D_k}{4}>\Gamma_*.
\]
In particular, the dominant set and its complement still have the separation required for the polynomial estimate.

If separation first failed at $N$, the recurrences would use separation only before $N$ and would still give strict separation at $N$, a contradiction. For $N=1$, the first transition is included in the lemma and gives the same bounds. The groups therefore keep alternating, and the common components stay between the outer groups at every cycle.

The initial inequalities are open and feasible, for example by placing the outer amplitudes near $D_0$ and $0$, and the common amplitudes near $D_0/2$. Nonzero polynomial values preserve all $n>m$ active components. A dependence among $h_k,Ah_k,\ldots,A^{m-1}h_k$ would give a polynomial of degree less than $m$ vanishing at all $n>m$ active eigenvalues. Thus the Krylov vectors are independent; the nonzero step sizes and Lemma~\ref{lem:history} give rank $m$ for every historical extraction.

It remains to convert this persistent separation into decay of the original gradient. The separation bounds the decrease of every component. Let $I_k=L\cup C$ for even $k$ and $I_k=C\cup R$ for odd $k$, and put $M_k=\max_{j\notin I_k}P_{j,k}$, and set $\Gamma_k=\min_{i\in I_k}P_{i,k}-M_k$. The bounds proved above give $\Gamma_k\ge D_k/4\ge c\sigma^k$ for some $c>0$. By \eqref{eq:oneSidedFilter}, for components in $I_k$ $\log|p_k(\lambda_i)|$ is at most $-2\Gamma_k+K_f$; for the other components its absolute value is bounded by the same constant $K_f$. As the outer groups exchange roles, every outer component belongs to one of $I_k,I_{k-1}$, while each common component belongs to both. The historical recurrence \eqref{eq:B}, with $\ell-1\ge1$, therefore gives $P_{i,k+1}-P_{i,k}\le-c_1\sigma^{k-1}+C_1$, with $c_1>0$, for every component and $k\ge1$. Summation proves the upper envelope. Conversely, for a component in $I_k$, $0<P_{i,k}-M_k\le D_k+2S_k+\max_{j\in C}|Y_{j,k}|$. Thus \eqref{eq:oneSidedFilter} bounds $|\log|p_k(\lambda_i)||$ by $C_2\sigma^k$. For components outside $I_k$ these logarithms remain bounded by a constant. Their total contribution is $O(N)$, while $\sum_{k<N}\sigma^k=O(\sigma^N)$. The product of the historical Ritz polynomials therefore gives the lower envelope. Lemma~\ref{lem:transfer} turns the two-sided bound into the limit \eqref{eq:B1rate} for all gradient updates. Since $\lambda_1\|x_t\|\le\|g_t\|\le\lambda_n\|x_t\|$, replacing $\|g_t\|$ by $\|x_t\|$ leaves the limit unchanged.
\end{proof}

These results give open sets of superlinearly convergent initial points for every $m+1<n\le2m$.

\section{Almost-everywhere convergence of the sampled method}\label{sec:A}

The previous two sections gave superlinear convergence results for the rolling method, first at $n=m+1$ and then for $m+1<n\le2m$. Those proofs use either the exact recurrence for the history or the alternation of two dominant groups of size $m$. They do not give a result for general $n>2m$. To reach higher dimensions, we use the sampled method of Section~\ref{sec:SCR}. At the start of each cycle, $m$ sampling updates rebuild the history, whose Krylov space is generated by the cycle-start gradient $F_k$. Thus $F_k$ determines the Ritz polynomial $p_k$ used in that cycle, and
\[
 F_{k+1}=d_k(A)p_k(A)^\ell F_k.
\]
The previous polynomial $p_{k-1}$ no longer appears. In this respect the new recurrence is like steepest descent: the current gradient determines the polynomial for the current cycle. For $m=1$, the reciprocal Ritz value from $F_k$ is the steepest descent step size. If we leave out the sampling factor in \eqref{eq:scalarSamplingInterpretation}, the endpoint recurrence is exactly that of CSD. For general $m$, the same gradient $F_k$ determines all $m$ Ritz steps. This lets us follow the CSD proof of \citet{GuCSD} by studying the Jacobian, the number of preimages of the log-ratio map, and the cumulative decrease. We also account for the sampling factor, which gives a bounded translation in the log ratios.

We first show how taking preimages under the log-ratio map changes volume. We then use this estimate to bound the measure of initial points for which some gradient component decreases too slowly. Outside a null set, each component decreases fast enough to give doubly exponential decay of the gradient norm. The resulting sufficient condition can be met for every $n>m$ by increasing $\ell$. We then count the allowed sequences of sampling lists to extend the estimate to choices based on the iteration history.

\subsection{Prescribed sampling steps}

Throughout this subsection, let $m\ge1$, $n>m$, and $\ell\ge2$. The sampling steps are chosen in advance. They are positive and independent of the initial point and the iterates. They can vary both within and between cycles. Assume that there are constants $0<\underline\gamma\le\overline\gamma<\infty$ such that
\begin{equation}\label{eq:samplingbound}
 \underline\gamma\le|1-\tau_{k,j}\lambda_i|\le\overline\gamma
 \quad\text{for every }k,j,i.
\end{equation}
Thus every sampling factor has uniformly bounded modulus and stays away from zero. Repeating one list with nonzero sampling factors satisfies this assumption. So does any list of positive steps chosen in advance and bounded above by $\bar\tau<1/\lambda_n$.

Fix a reference component $r$ and set $z_i=\log(u_i/u_r)$ ($i\ne r$), where $u_i=(F_k)_i^2$. The Ritz updates alone give
\[
 T(z)_i=z_i+2\ell\log\left|\frac{\pi(\lambda_i;z)}{\pi(\lambda_r;z)}\right|,
 \qquad i\ne r.
\]
Here $\pi(\cdot;z)$ is the monic polynomial given by \eqref{eq:orthogonal} for the current weights. The factor $\pi(0)$ cancels in the ratio. Write $d=n-1$ and let $\vol$ denote Lebesgue measure in these coordinates. We begin with full-support gradients. Initial gradients with a zero component lie in a finite union of coordinate hyperplanes, which has measure zero. Even with full support initially, a Ritz polynomial could vanish at an eigenvalue in a later cycle and set that component to zero. We compute $T$ where none of its polynomial values vanish. Below we show that the initial points that reach a zero polynomial value in finite time also form a null set.

To find the volume of a preimage, we need the volume factor of each local inverse and the number of preimages of each target. \citet{GuCSD} used these two properties for the map with one step size. The calculation here allows $m$ Ritz values.

First consider the Jacobian. The polynomial has only $m$ free coefficients, so the nonlinear part of $T$ can be split into two maps. The map $R(z)$ gives the vector $c$ of the $m$ coefficients below the leading term. The map $L(c)$ gives the log ratios $L(c)_i=2\log|\pi_c(\lambda_i)/\pi_c(\lambda_r)|$. Thus $T(z)=z+\ell L(R(z))$. Differentiating \eqref{eq:orthogonal} with respect to each $z_i$ gives $DR$, and differentiating the log ratios with respect to $c$ gives $DL$. Multiplying the two derivative matrices and using \eqref{eq:orthogonal} again gives $DR\,DL=-2I_m$. The determinant identity for the two orders of the matrix product now gives
\[
 \det DT=\det(I_d+\ell DL\,DR)
 =\det(I_m+\ell DR\,DL)=(1-2\ell)^m.
\]
Lemma~\ref{lem:jacA} gives the full calculation. Each local inverse therefore multiplies volume by $(2\ell-1)^{-m}$.

Next count the preimages. A polynomial satisfying \eqref{eq:orthogonal} with nonzero values at all eigenvalues has its $m$ roots in $m$ distinct open gaps $(\lambda_j,\lambda_{j+1})$. There are $\binom{n-1}{m}$ choices of gaps. To see why each choice gives exactly one preimage, fix positive target weights $U_i$. The equation for the polynomial before the update is the stationarity condition for $\sum_i U_i/\pi_c(\lambda_i)^{2\ell-2}$ as a function of $c$. On the coefficient region for the chosen gaps this function is strictly convex and tends to infinity at the boundary. Lemma~\ref{lem:inversePolynomial} proves these facts and gives a unique minimizer in each region. Lemma~\ref{lem:branches} then recovers the source weights up to scale from the polynomial. Each choice of gaps gives one preimage, and different choices give different preimages.

Multiplying the number of preimages by the common volume factor of each local inverse gives the total factor
\[
 \rho=\frac{\binom{n-1}{m}}{(2\ell-1)^m}.
\]
We now include the sampling updates.
For \algSCR\ the actual map of the weight ratios in cycle $k$ is
\begin{equation}\label{eq:translated}
 T_k(z)=T(z)+\Delta_k,\qquad
 (\Delta_k)_i=2\log\left|\frac{d_k(\lambda_i)}{d_k(\lambda_r)}\right|.
\end{equation}
The translation is independent of $z$ and uniformly bounded by \eqref{eq:samplingbound}, so $T_k$ has the same inverse-image volume factor $\rho$. For $\Psi_N=T_{N-1}\circ\cdots\circ T_0$, with $\Psi_0(z)=z$,
\begin{equation}\label{eq:nonautonomous}
 \vol(\Psi_N^{-1}E)=\rho^N\vol(E).
\end{equation}
When $\rho<1$, the $N$-cycle preimage of a fixed bounded set has volume $\rho^N$ times the volume of the set. To use this fact for convergence, we need a region that contains the log ratios when a gradient component decreases too slowly. If the reference component decreases too little and the other components avoid very small polynomial values, the log ratios stay in a box whose volume we can bound. The box grows with $N$, so we need to compare this growth with the decay of $\rho^N$.

To measure the decrease of each component by a nonnegative quantity, we rescale the spectral interval to $[0,1]$. Put $\Delta_\lambda=\lambda_n-\lambda_1$ and $\beta_i=(\lambda_i-\lambda_1)/\Delta_\lambda$. Let $\widehat\pi$ be the monic polynomial satisfying \eqref{eq:orthogonal} with $\lambda_i$ replaced by $\beta_i$ and with the same weights. Its numerator at each $\beta_i$ is a nonzero polynomial in the positive weight ratios, as shown in Appendix~\ref{app:samplingCancellation}. The zero sets are null in the log coordinates. Equation~\eqref{eq:nonautonomous} implies that their finite-time inverse images are also null, so we can exclude their countable union.

Each root $\widehat\theta_j$ of $\widehat\pi$ lies in $(0,1)$, so every factor $|\beta_i-\widehat\theta_j|$ is less than one. Since no polynomial value vanishes along this sequence, we obtain $0<|\widehat\pi(\beta_i)|=\prod_j|\beta_i-\widehat\theta_j|<1$. Define the nonnegative quantities
\begin{equation}\label{eq:samplingLoss}
 a_{i,k}=-\log|\widehat\pi_k(\beta_i)|,\qquad
 S_{r,N}=\sum_{k=0}^{N-1}a_{r,k}.
\end{equation}
The original polynomial factor is $p_k(\lambda_i)=\Delta_\lambda^m\widehat\pi_k(\beta_i)/\pi_k(0)$. Since $\lambda_1^m\le|\pi_k(0)|\le\lambda_n^m$, its logarithm differs from $-a_{i,k}$ by a uniformly bounded term. Thus the cumulative logarithms of the original polynomial factors equal $-S_{i,N}+O(N)$.
To see how these decreases change the relative component sizes, take the ratio of component $i$ and component $r$ in \eqref{eq:A}, square it, and take logarithms. The common normalization $\pi_k(0)$ cancels, giving
\begin{equation}\label{eq:samplingRatio}
 z_{i,k+1}-z_{i,k}
   =2\ell(a_{r,k}-a_{i,k})+(\Delta_k)_i.
\end{equation}

We now want to show that for almost every initial point, $S_{r,N}$ eventually exceeds $M_0e^{\kappa N}$ for $M_0>0$ and for the positive exponents $\kappa$ specified below. Consider the initial points for which $S_{r,N}\le M_0e^{\kappa N}$. Since the terms $a_{i,k}$ are nonnegative, this bound also controls the cumulative decrease of the reference component at each earlier cycle. Summing \eqref{eq:samplingRatio} then gives an upper bound on every log ratio.

A lower bound requires more care. If $\widehat\pi_k(\beta_i)$ is close to zero at another component, then $a_{i,k}$ is large and $z_{i,k}$ can drop sharply. This near cancellation helps component $i$ decrease, but gives no lower bound on the decrease of component $r$. It can also take the log ratios below the box we want to use. We therefore bound the measure of the initial points that lead to such a drop.

Appendix~\ref{app:samplingEvents} separates two cases. If no near cancellation reaches the chosen threshold before cycle $N$, the upper bound and the stepwise lower bound keep the trajectory in a box of volume at most $CN^de^{\kappa dN}$. Otherwise we record the first cycle at which the threshold is reached. We choose the threshold so that the corresponding near-cancellation set has volume at most $Ce^{-N}$. The appendix gives the threshold, the side lengths of the box, and the polynomial sublevel estimates. For one prescribed sequence, the volume identity then bounds the measure of the event of small cumulative decrease by
\[
 C N^d(\rho e^{\kappa d})^N
 +C e^{-N}\sum_{k=0}^{N-1}\rho^k.
\]
This bound is summable in $N$ if $\rho<1$ and $\rho e^{\kappa d}<1$. By Borel--Cantelli, almost every initial point belongs to only finitely many of these events. We will use this fact to bound the cumulative decrease of each component and hence the gradient norm.

The next lemma states this estimate and also allows several prescribed sampling sequences for use in the next subsection. If there are at most $C_he^{hk}$ distinct prefixes of length $k$, taking their union adds that factor to the bound at cycle $k$. The sufficient condition becomes $e^h\rho e^{\kappa d}<1$. For a single prescribed sequence we take $h=0$.

\begin{lemma}[Summable sets of small cumulative decrease]\label{lem:samplingBadEvents}
Use the log ratios, sampling bounds, and cumulative quantities defined above, and suppose $\rho<1$. Let a family of prescribed sampling sequences have at most $C_h e^{hk}$ distinct prefixes of length $k$, with $h\ge0$ and $C_h\ge1$. Fix a reference $r$, $R\ge1$, $M_0>0$, and
\[
 0<\kappa<\frac{\log(1/\rho)-h}{d}.
\]
Let $\mathcal E_N$ be the set of $z_0\in[-R,R]^d$ for which $S_{r,N}\le M_0e^{\kappa N}$ along at least one sequence in the family. Then $\sum_{N\ge1}\vol(\mathcal E_N)<\infty$. The constants in the measure bound depend on the common sampling bounds, $r,R,M_0,\kappa,h,C_h$, and the spectrum and algorithm parameters. For a single prescribed sequence, take $h=0$ and $C_h=1$.
\end{lemma}
\begin{proof}\hspace{0.5em}
There are finitely many prefixes of each length, so their finite-time zero sets form a countable null union. Exclude this union. Lemma~\ref{lem:samplingVolume} splits the event into two cases: the trajectory stays in a bounded box through time $N$, or it first reaches a near-cancellation set at an earlier time.

For the first case, the box volume is at most $CN^de^{\kappa dN}$. There are at most $C_he^{hN}$ prefixes, and the inverse-image volume factor for each is $\rho^N$. Their product is at most
\[
 C N^d(e^h\rho e^{\kappa d})^N.
\]
For the second case, the near-cancellation sets can be chosen to have volume at most $Ce^{-N}$. At time $k$ the total inverse-image factor over all prefixes is at most $C_h(e^h\rho)^k$. Summing over the possible first-entry times gives
\[
 Ce^{-N}\sum_{k=0}^{N-1}(e^h\rho)^k
 \le \frac{Ce^{-N}}{1-e^h\rho}.
\]
The assumed bound on $\kappa$ gives $e^h\rho e^{\kappa d}<1$, and therefore also $e^h\rho<1$. Both contributions are summable in $N$. All constants are uniform over the prescribed sequences in the family.
\end{proof}

\Needspace{5\baselineskip}
\begin{theorem}[Almost-everywhere superlinear convergence]\label{thm:A1}
Let $m\ge1$, $n>m$, and $\ell\ge2$. Apply \algSCR\ with positive sampling steps prescribed independently of the initial point and trajectory, and suppose that \eqref{eq:samplingbound} holds. If
\begin{equation}\label{eq:Acondition}
 (2\ell-1)^m>\binom{n-1}{m},
 \qquad
 \kappa_0=\frac1{n-1}\log\frac{(2\ell-1)^m}{\binom{n-1}{m}},
\end{equation}
then there is a full-Lebesgue-measure subset of $\R^n$ on which, for every $0<\kappa<\kappa_0$, there is $c=c(x_0,\kappa)>0$ such that all sufficiently late actual gradients of \algSCR\ satisfy
\begin{equation}\label{eq:Arate}
 \|g_t\|\le\exp\left[-c\exp\left(\frac{\kappa t}{m(\ell+1)}\right)\right].
\end{equation}
The corresponding iterates converge R-superlinearly to the origin. The index $t$ counts all $m(\ell+1)$ updates per cycle, including sampling.
\end{theorem}
\begin{proof}\hspace{0.5em}
Apply Lemma~\ref{lem:samplingBadEvents} to the single prescribed sampling sequence, with $h=0$. For each reference $r$, $R\ge1$, $M_0>0$, and $0<\kappa<\kappa_0$, the events
\[
 \mathcal E_N=\{z_0\in[-R,R]^d:S_{r,N}\le M_0e^{\kappa N}\}
\]
have summable measures. We use this to prove that $S_{r,N}/e^{\kappa N}\to\infty$ for every component, then return to the original gradients.

First fix $r,R,M_0,\kappa$. By Borel--Cantelli, almost every point in the initial box $[-R,R]^d$ belongs to only finitely many $\mathcal E_N$. Hence $S_{r,N}>M_0e^{\kappa N}$ for all sufficiently large $N$ on a full-measure subset of that box.

To get one set that works for every reference component and every exponent in the stated range, choose a sequence $0<\kappa_j\uparrow\kappa_0$, let $R,M_0$ range over the positive integers, and let $r$ run over all reference components. Each choice of parameters excludes a null subset of the corresponding box. Express these sets in one system of log-ratio coordinates and take their union. The union is null because a change of reference preserves null sets and there are only countably many parameter choices.

The boxes with integer radii cover the log-ratio space, and the integer thresholds $M_0$ can be arbitrarily large. Thus on the same full-measure set, $S_{r,N}/e^{\kappa_jN}\to\infty$ for every $r$ and $j$. For any $0<\kappa<\kappa_0$, we can choose $\kappa_j>\kappa$ and obtain $S_{r,N}/e^{\kappa N}\to\infty$ as well.

We now translate this cumulative decrease into the rate of the original gradients. The physical polynomial satisfies $\pi_k(\lambda_i)=\Delta_\lambda^m\widehat\pi_k(\beta_i)$ and $\lambda_1^m\le|\pi_k(0)|\le\lambda_n^m$. Equation~\eqref{eq:A} and \eqref{eq:samplingbound} therefore give $\log|(F_N)_i|=\log|(F_0)_i|-\ell S_{i,N}+O(N)$. The term $O(N)$ is small compared with $S_{i,N}$, so every component of $F_N$ decays doubly exponentially. To transfer it to individual updates, put $L_{\rm cyc}=m(\ell+1)$ and $K=\max\{1,\overline\gamma,M\}$, with $M$ from Lemma~\ref{lem:transfer}. Both sampling and Ritz updates have norm at most $K$, uniformly in the cycle. Hence $\|g_{NL_{\rm cyc}+j}\|\le K^{L_{\rm cyc}}\|F_N\| \qquad(0\le j<L_{\rm cyc})$. Since $t/L_{\rm cyc}\le N+1$ in this block, a change of the positive constant in the endpoint estimate gives \eqref{eq:Arate}, including every sampling update. Since $\|x_t\|\le\lambda_1^{-1}\|g_t\|$, the iterates have the stated R-superlinear convergence.

Finally, the exceptional set was measured in log-ratio coordinates. On each orthant of the initial-gradient space, $s=|g_{r,0}|>0$ and $z_i=\log(g_{i,0}^2/g_{r,0}^2)$, $i\ne r$, are smooth coordinates with a smooth inverse. Fubini therefore transfers the log-ratio null set to a null set of initial gradients in $\R^n$. The invertible map $g_0=Ax_0$ gives the result for initial points.
\end{proof}

For $m=1$, the sufficient condition reduces to $2\ell>n$. For $m=2$ and $\ell=2$, it allows $n=5$, since $9>\binom42=6$; this lies beyond the range of the rolling results above.

\begin{remark}[Sampling steps chosen in advance or adaptively]
The steps $\tau_{k,j}$ need not be equal, but the present theorem requires all of them to be chosen independently of the iterates. If the sampling steps depend on $z$, then $\Delta_k=\Delta_k(z)$ and the derivative in \eqref{eq:translated} contains $D\Delta_k$. Condition~\eqref{eq:samplingbound} alone no longer proves the constant Jacobian or inverse-branch count.
\end{remark}

\subsection{Adaptive choice among finitely many sampling lists}\label{sec:adaptiveA}

The preceding subsection proves almost-everywhere doubly exponential convergence under $\rho<1$ for a sampling sequence chosen in advance. In practice, we may want to adjust the sampling steps using gradients already computed. Different lists give different sampling polynomials $d_k$. This changes the gradient after the sampling updates and can also affect the conditioning of the history matrix. To allow some choice while keeping a convergence guarantee, we prepare finitely many candidate lists before the iteration and choose among them during the run.

Fix $R$ candidate lists. Each is an ordered list of $m$ positive step sizes used for the first $m$ updates of a cycle. The cycle then takes $m\ell$ Ritz updates. All lists satisfy \eqref{eq:samplingbound} with the same constants, and $\ell$ stays fixed. At the start of a prescribed decision cycle, we choose a list using the history computed so far. We use that list in every cycle until the next decision. Each cycle still takes $m$ sampling updates, even when no new list is chosen.

The decision cycles can be spaced in several ways. If a new choice is allowed in every cycle, there are at most $R^N$ sequences of lists for the first $N$ cycles. If a choice is made every $J$ cycles and held in between, there are at most $R^{\lceil N/J\rceil}$ sequences. With choices only at cycles $0,1,2,4,8,\ldots$, there are $O(\log N)$ decisions in the first $N$ cycles, so the number of sequences grows at most as a power of $N$. Sparse decisions mean that we choose the list again in fewer cycles. The number of sampling updates per cycle remains the same. With the geometrically spaced decisions the gaps between decisions grow with the cycle index.

The number of allowed sequences affects the measure estimate. In the preceding subsection, fixing a sampling sequence gave a known map whose preimage volumes could be used to bound the measure of initial points with small cumulative decrease. When the list depends on the iteration history, the map itself depends on the trajectory. The constant Jacobian and the preimage volume formula for a prescribed sequence cannot be applied directly. To use Borel--Cantelli here, we need to establish summability for the events under adaptive choice.

We take the union over the allowed sequences. Write $\mathcal W_N$ for the family of all allowed sequences of lists for the first $N$ cycles. It includes every sequence that any allowed rule can produce. Fix the reference component $r$, the initial box radius $B\ge1$, the threshold $M_0>0$, and the exponent $\kappa>0$. For each $w\in\mathcal W_N$, fix the lists in $w$ and define
\[
 \mathcal E_N(w)=\{z_0\in[-B,B]^d:
 S_{r,N}(w)\le M_0e^{\kappa N}\},
 \qquad
 \widetilde{\mathcal E}_N=\bigcup_{w\in\mathcal W_N}\mathcal E_N(w).
\]
The union is fixed before we choose a particular rule. If its measures are summable in $N$, the exceptional null set works for all allowed choices based on the history.

Taking the union brings the number of sequences into the volume bound. The terminal box contributes at most $|\mathcal W_N|\rho^N CN^de^{\kappa dN}$. For first entry into the near-cancellation set at cycle $k$, we multiply its volume by $|\mathcal W_k|\rho^k$ and sum over $k$. If $|\mathcal W_k|\le C_he^{hk}$, the two contributions contain the factors $(e^h\rho e^{\kappa d})^N$ and $\sum_{k<N}(e^h\rho)^k$, respectively. The sufficient condition in Lemma~\ref{lem:samplingBadEvents} is thus $e^h\rho e^{\kappa d}<1$. This leads us to use the exponential growth rate of the number of allowed sequences,
\begin{equation}\label{eq:entropy}
 h_{\mathrm{seq}}=\limsup_{N\to\infty}\frac{\log|\mathcal W_N|}{N}.
\end{equation}
A single prescribed sequence and geometrically spaced decisions both give $h_{\mathrm{seq}}=0$. Free choice among all $R$ lists gives $h_{\mathrm{seq}}=\log R$ when it is allowed every cycle and $h_{\mathrm{seq}}=(\log R)/J$ when it is allowed every $J$ cycles. Thus the sufficient condition for free choice at every cycle is $R\rho<1$, while geometrically spaced decisions keep the condition $\rho<1$. The next theorem covers these cases.

\Needspace{5\baselineskip}
\begin{theorem}[Adaptive selection from finitely many sampling lists]\label{thm:adaptiveA}
Let $\rho=\binom{n-1}{m}/(2\ell-1)^m$, with $\ell\ge2$. Under the assumptions on the finite set of lists above, suppose
\begin{equation}\label{eq:adaptivecondition}
 h_{\mathrm{seq}}+\log\rho<0.
\end{equation}
Then a common full-measure set in $\R^n$ has doubly exponential convergence for every allowed adaptive selection rule and every $0<\kappa<\frac{\log(1/\rho)-h_{\mathrm{seq}}}{n-1}$. On this set, \eqref{eq:Arate} holds with a constant and a starting index that depend on the initial point, exponent, and rule. In each of the following three cases, at any decision cycle we may choose any of the $R$ candidate lists using the history computed so far. The choice stays in use until the next decision.
\begin{enumerate}
\item Every cycle is a decision cycle. The conclusion holds if $(2\ell-1)^m>R\binom{n-1}{m}$.
\item Decisions are made every $J$ cycles. The conclusion holds if $(2\ell-1)^m>R^{1/J}\binom{n-1}{m}$.
\item Decisions are made at cycles $0,1,2,4,8,\ldots$. Then $h_{\mathrm{seq}}=0$, so the condition and rate range are the same as in Theorem~\ref{thm:A1}.
\end{enumerate}
\end{theorem}
\begin{proof}\hspace{0.5em}
Fix $\kappa$ in the stated range and choose $\eta>0$ such that $h_{\mathrm{seq}}+\eta+\log\rho+\kappa(n-1)<0$. There is $C_\eta$ with $|\mathcal W_k|\le C_\eta e^{(h_{\mathrm{seq}}+\eta)k}$ for every $k$. Lemma~\ref{lem:samplingBadEvents}, with $h=h_{\mathrm{seq}}+\eta$, shows that the events $S_{r,N}(w)\le M_0e^{\kappa N}$ for at least one $w\in\mathcal W_N$ have summable measures. The lemma takes the union over sequences before measuring the event, so it covers every history-dependent rule on the same exceptional set.

Borel--Cantelli therefore gives, outside a null set, an $N_0$ such that for every $N\ge N_0$ and every allowed sequence $w\in\mathcal W_N$, $S_{r,N}(w)>M_0e^{\kappa N}$. Here $S_{r,N}(w)$ is the sum of the reference component's negative logarithms along that sequence. This uniform statement covers every allowed selection rule, since the union was taken before choosing a rule.

Take countable intersections over integer initial-box radii, positive integers $M_0$, all reference components, and exponents increasing to $[\log(1/\rho)-h_{\mathrm{seq}}]/(n-1)$. Integer thresholds are enough because each finite positive threshold is bounded by one of them. The resulting full-measure set works for every allowed rule and every exponent in the stated interval: a larger exponent from this sequence gives the bound for any requested smaller one. A spectral component becoming zero at a finite time depends on only a finite prefix of sampling lists. There are countably many such prefixes over the finite set of lists, so these sets form a countable null union. Uniform sampling bounds and the number $m(\ell+1)$ of updates per cycle give the same gradient-norm estimates, including the updates within each cycle, for every allowed rule.

The numbers of possible sequences in the three special cases are at most $R^N$, $R^{\lceil N/J\rceil}$, and $R^{2+\log_2(N+1)}$, respectively. The last is polynomial in $N$, so $h_{\mathrm{seq}}=0$.
\end{proof}

\begin{remark}
For example, $n=5,m=2,\ell=3$ gives $\rho=6/25$; arbitrary history-dependent choices between two sampling lists satisfying \eqref{eq:samplingbound} are covered since $2\rho<1$. More generally, whenever the original strict condition $\rho<1$ holds, a sufficiently large fixed holding length $J$ allows any finite set of lists chosen in advance. The third rule keeps the rate range of Theorem~\ref{thm:A1}, with $m(\ell+1)$ updates per cycle.
\end{remark}

Suppose a spectral upper bound $L\ge\lambda_n$ is known. We can set up the lists in advance by requiring $0<\tau_j\le c/L$ for each step, with $0<c<1$. This gives $1-c\le1-\tau_j\lambda_i\le1$ for all $i$, so every list meets the common sampling bound without knowing the individual eigenvalues. For the same full-support cycle-start gradient $F_k$, the histories from different lists span the same Krylov space $\K_m(A,F_k)$. They therefore give the same Ritz values in exact arithmetic. Their sampling polynomials $d_k$ differ, however, so the lists change the gradient at the end of this cycle and hence the later trajectory. In finite precision they can also affect the conditioning of the history matrix.

In the experiment, we set up two lists of length $m$: $(0.5/L,\ldots,0.5/L)$ and $(0.9/L,\ldots,0.9/L)$. SCR starts with the first list at cycle $0$. At the start of cycle $k=1,2,4,8,\ldots$, both $F_k$ and $F_{k-1}$ are already available. If $\|F_k\|>\|F_{k-1}\|/2$, SCR uses the list with step $0.9/L$. Otherwise it uses the list with step $0.5/L$. At all other cycles it keeps the last list without comparing the norms again. This is a simple rule that uses the larger sampling step when the gradient norm has not halved over the preceding cycle. The decision times give $h_{\mathrm{seq}}=0$, so Theorem~\ref{thm:adaptiveA} gives the same sufficient condition and range of rate exponents as for prescribed sampling.

We compare this SCR implementation with RCR in settings that meet the SCR condition and lie outside the RCR ranges proved in Theorems~\ref{thm:B1}, \ref{thm:B2}, and~\ref{thm:overlap}. For $m\in\{1,2\}$ and $n\in\{3,4,5\}$, we take the smallest $\ell\ge2$ satisfying \eqref{eq:Acondition}. We exclude $n=m+1$ and also $m+1<n\le2m$ when $\ell\ge4$. This gives the five rows in Table~\ref{tab:illustration}. The test problems have geometric spectra $\lambda_i=\kappa^{(i-1)/(n-1)}$, with $\kappa=10^2,10^4$ and ten paired random starts per spectrum. We use a random orthogonal eigenbasis and independent log amplitudes uniform on $[-1,1]$, assign random signs, and scale $g_0$ to have norm one. Both methods use the same matrices and starts, and we reuse the problems with the same $n$ across memory lengths. In the experiment $L=\lambda_n$, and RCR starts with $m$ updates of step size $0.5/L$. Both methods reach the tolerance on all 100 paired tests. SCR uses fewer updates on 68 tests, ties on three, and uses more on 29. Its median update count is lower in four of the five rows.

\begin{table}[!ht]
\centering
\begin{tabular*}{\linewidth}{@{\extracolsep{\fill}}rrrrrrrr@{}}
\toprule
 & & & \multicolumn{2}{c}{Successes} & \multicolumn{2}{c}{Median updates} & \\
\cmidrule(lr){4-5}\cmidrule(lr){6-7}
$m$ & $n$ & $\ell$ & SCR & RCR & SCR & RCR & SCR wins/ties/losses\\
\midrule
1 & 3 & 2 & 20 & 20 & 32 & 41 & 16/0/4 \\
1 & 4 & 3 & 20 & 20 & 48 & 45.5 & 7/1/12 \\
1 & 5 & 3 & 20 & 20 & 66 & 80 & 12/0/8 \\
2 & 4 & 2 & 20 & 20 & 28 & 40 & 19/0/1 \\
2 & 5 & 2 & 20 & 20 & 58 & 73.5 & 14/2/4 \\
\bottomrule
\end{tabular*}
\caption{Comparison of SCR and RCR on 20 paired problems per row, at $\|g_t\|/\|g_0\|\le10^{-8}$. SCR selects lists using gradient norms at cycles $1,2,4,8,\ldots$. Update counts include initialization and sampling. Medians and wins/ties/losses use common successes.}
\label{tab:illustration}
\end{table}

\section{Conclusions}

Our results extend the superlinear convergence analysis of cyclic methods with historical Ritz steps to higher spectral dimensions. At $n=m+1$, the first spectral dimension beyond finite termination, the rolling method has an exact doubly exponential rate per gradient update outside an explicit exceptional set. For $m+1<n\le2m$, the same rate follows on nonempty open sets when $\ell\ge4$. The separation of the initial spectral components lets the dominant group change from one cycle to the next, so that every component is strongly reduced over successive cycles.

To reach larger spectral dimensions, we change how the history is formed. Rebuilding it with $m$ sampling updates per cycle lets the cycle-start weights determine the Ritz polynomial. For any $n>m$, sufficiently many repetitions then give doubly exponential convergence for almost every initial point under the prescribed sampling bounds. The guarantee also holds when the sampling list is chosen from a finite set using the past iterates at suitably spaced decision times.

This leaves a question for the rolling method. Can increasing the repetition count extend its guarantees beyond $n=2m$, or enlarge the sets of initial points for which convergence is superlinear within that range? The almost-everywhere question for three-dimensional BB1 also remains open. For the sampled method, the wider theoretical range makes the choice of sampling steps and repetition count important in practice, where rebuilding the history costs extra updates and the Ritz values are computed in finite precision.

\backmatter

\noindent\textbf{Statements and Declarations}\par
\smallskip

\noindent\textbf{Competing interests.} The author declares no competing interests.

\noindent\textbf{Data availability.} The MATLAB implementation and paired-run results underlying Table~\ref{tab:illustration} are provided in the supplementary material.

\begin{appendices}

\section{Polynomial estimates and constants for rolling reuse}\label{app:constants}

We give explicit constants for the polynomial bounds and recurrences used in the initial separation thresholds in Theorems~\ref{thm:B2} and~\ref{thm:overlap}, using finite subset sums. These analysis constants depend on the spectrum; the algorithm uses stored gradients.

\subsection{A one-sided polynomial estimate}

For a defining gradient $h$, set $u_i=h_i^2$. For a subset $S\subseteq\{1,\ldots,n\}$, use $u_S=\prod_{i\in S}u_i$ and $\Delta_S=\prod_{i<j,\ i,j\in S}(\lambda_j-\lambda_i)$, with every empty product equal to one. The monic polynomial defined by \eqref{eq:orthogonal} has the finite subset expansion
\begin{equation}\label{eq:subset}
 \pi(t)=\frac{\displaystyle\sum_{|S|=m}
 u_S\Delta_S^2\prod_{j\in S}(t-\lambda_j)}
 {\displaystyle Z_m(u)},\qquad
 Z_m(u)=\sum_{|S|=m}u_S\Delta_S^2.
\end{equation}
After normalization at zero, this is the complete-sweep Cauchy--Binet representation in \citep[Proposition~11]{YangYuanLMSD}. The identity is useful here because the leading terms have the same sign when one spectral group lies entirely to one side of the other. This prevents cancellation in the polynomial bound of Lemma~\ref{lem:twobound}. We verify the identity and give the constants below.

In the subset formula \eqref{eq:subset}, the numerator has leading coefficient $Z_m(u)$, so the polynomial is monic. To check \eqref{eq:orthogonal}, we change the order of the sums over $i$ and $S$ in $\sum_i u_i\lambda_i^a\pi(\lambda_i)$. Group the terms by the $(m+1)$-element set $T=S\cup\{i\}$. Their sum is proportional to $\sum_{i\in T}\lambda_i^a/\prod_{j\in T\setminus\{i\}} (\lambda_i-\lambda_j)=0$ for $a<m$, by Lagrange interpolation. Thus \eqref{eq:subset} follows from uniqueness in \eqref{eq:orthogonal}.

\begin{lemma}[Polynomial bound for a one-sided spectral group]\label{lem:twobound}
There are constants $\Gamma_*>0$ and $C$ depending only on the spectrum and $m$ with the following property. Let $K$ consist of $m$ spectral points and let its nonempty complement $L$ lie entirely on one side of $K$ in the ordered spectrum. If
\[
 \Gamma=\min_{i\in K}P_i-\max_{j\in L}P_j\ge\Gamma_*,
 \qquad P_i=\tfrac12\log u_i,
\]
then, with $M_L=\max_{j\in L}P_j$,
\begin{equation}\label{eq:filterbound}
 \log|p(\lambda_i)|=
 \begin{cases}2(M_L-P_i)+e_i,&i\in K,\\e_i,&i\in L,
 \end{cases}\qquad |e_i|\le C.
\end{equation}
Every $p(\lambda_i)$ is nonzero. The bound is independent of the ratios of weights within either group.
\end{lemma}
\begin{proof}\hspace{0.5em}
We compare the leading subset terms with the remaining terms, and then normalize the polynomial values. The constants below bound these two contributions separately. For every $m$-element subset $S$, put $z_S=\Delta_S^2$, $Z=\sum_Sz_S$, and $c_{iS}=z_S\prod_{j\in S}(\lambda_i-\lambda_j)$. The one-sided condition means that $K$ consists of either the first $m$ or the last $m$ support points. Take extrema over these two choices so that the same constants apply to both groups. For $i\in K$ define
\[
\begin{aligned}
 \mathcal L_{Ki}&=\{(K\setminus\{i\})\cup\{j\}:j\in L\},\qquad
 a_{Ki}=\min_{S\in\mathcal L_{Ki}}|c_{iS}|,\\
 A_{Ki}&=\sum_{S\in\mathcal L_{Ki}}|c_{iS}|,\qquad
 R_{Ki}=\sum_{\substack{|S|=m,\ i\notin S\\S\notin\mathcal L_{Ki}}}|c_{iS}|.
\end{aligned}
\]
For $i\in L$, define $a_{Ki}=A_{Ki}=|c_{iK}|$ and $R_{Ki}=\sum_{S\ne K}|c_{iS}|$, with zero terms harmless. All $a_{Ki}$ are positive. With ratios involving $R_{Ki}=0$ omitted, set
\[
 \epsilon_* =\min\left\{\frac12,\min_{K,i:R_{Ki}>0}
                         \frac{a_{Ki}}{2R_{Ki}}\right\},
 \qquad\Gamma_* =\frac12\log(1/\epsilon_*).
\]
Lower ratio bounds use the upper denominator estimate $Zu_K$, while upper ratio bounds use the lower estimate $z_Ku_K$. Valid lower and upper multiplicative constants for the normalized polynomial bound are
\[
 l_{Ki}=\frac{a_{Ki}}{2Z\lambda_n^m},\qquad
 u_{Ki}=\frac{A_{Ki}+R_{Ki}}{z_K\lambda_1^m},\qquad
 C=\max_{K,i}\{|\log l_{Ki}|,|\log u_{Ki}|\}.
\]
For $i\in K$, these constants multiply $\max_Lu/u_i$; for $i\in L$ they bound $|p(\lambda_i)|$ directly. To check the normalization, write $u_*=\max_{j\in L}u_j$ and $\epsilon=u_*/\min_{j\in K}u_j=e^{-2\Gamma}$. For $i\in K$, divide the numerator of \eqref{eq:subset} at $\lambda_i$ by $u_{K\setminus\{i\}}$. Its leading terms have the same sign and their absolute sum lies between $a_{Ki}u_*$ and $A_{Ki}u_*$. Every remaining nonzero term makes at least one additional replacement of a weight in $K$ by a weight in $L$, so the remainder has absolute value at most $R_{Ki}\epsilon u_*$. Since $\epsilon\le\epsilon_*$, the full absolute numerator therefore lies between $(a_{Ki}/2)u_{K\setminus\{i\}}u_*$ and $(A_{Ki}+R_{Ki})u_{K\setminus\{i\}}u_*$. For $i\in L$, division by $u_K$ instead gives a leading term of absolute value $a_{Ki}=A_{Ki}$ and a remainder of absolute value at most $R_{Ki}\epsilon$. The denominator lies between $z_Ku_K$ and $Zu_K$, and $\lambda_1^m\le|\pi(0)|\le\lambda_n^m$. Taking logarithms gives \eqref{eq:filterbound}.
\end{proof}

\subsection{Recurrences for disjoint groups}\label{app:rollingRecurrences}

Throughout this subsection, $n=2m$, $m\ge2$, and $\ell\ge4$. RCR has full-support initial data and is initialized by $m$ updates with the same $\tau>0$, where $d(\lambda_i)=(1-\tau\lambda_i)^m\ne0$. Let $I=\{1,\ldots,m\}$ and $J=\{m+1,\ldots,2m\}$. Put $P_{i,k}=\log|h_{i,k}|$ and $\bar P_{K,k}=m^{-1}\sum_{i\in K}P_{i,k}$ for $K=I,J$. The constants $\Gamma_*,C$ are from Lemma~\ref{lem:twobound}, chosen uniformly for the two groups.

Set $K_k=I$ for even $k$ and $K_k=J$ for odd $k$, and let $L_k$ be the other group. Define
\[
\begin{gathered}
 T_k=\bar P_{K_k,k},\quad U_k=\bar P_{L_k,k},\quad D_k=T_k-U_k,\\
 \quad\delta_{i,k}=P_{i,k}-\bar P_{\text{group containing }i,k},
 \quad S_k=\max_i|\delta_{i,k}|.
\end{gathered}
\]
Also write $M_{L_k,k}=\max_{i\in L_k}P_{i,k}$ for the largest log amplitude in $L_k$ at cycle $k$. Define $\Gamma_k=\min_{i\in K_k}P_{i,k}-\max_{j\in L_k}P_{j,k}$. Then $\Gamma_k\ge D_k-2S_k$.

The gap and spread will have different growth scales. For $\ell\ge4$, set
\begin{equation}\label{eq:scales}
 a=2\ell-3,\qquad B=2\ell-1=a+2,\qquad \nu=\sqrt B,
 \qquad\sigma=\sig,\qquad \sigma_-=2/\sigma.
\end{equation}
For $\ell\ge4$, the dominant and subdominant roots of $z^2-az+2$ are $\sigma$ and $\sigma_-$, respectively. They are the negatives of the roots in \eqref{eq:characteristic}, and they obey
\begin{equation}\label{eq:separation}
 0<\sigma_-<1<\nu<\sigma.
\end{equation}
Indeed $a\ge5$, and the polynomial is negative at $\sqrt{a+2}$ since $a\sqrt{a+2}>a+4$.

With the scales in \eqref{eq:scales}, put
\[
\begin{gathered}
 K_0=\ell C+\max_i|\log|d(\lambda_i)||,\qquad
 C_0=1+2K_0+\frac{4K_0}{B-1},\\
 C_s=(a+1)C_0+2K_0,\quad C_e=(a+3)C_s+2K_0,\quad
 C_f=\frac{\sigma}{\sigma-\sigma_-},\quad
 C_d=\frac{C_eC_f}{\sigma-\nu}.
\end{gathered}
\]
\begin{lemma}[Component updates and the gap recurrence]\label{lem:rollingRecurrences}
Let the full-support RCR trajectory have $\Gamma_k\ge\Gamma_*$ for $0\le k<N$, with $N\ge1$. The gap satisfies
\[
 D_1=aD_0+E_0,\qquad D_{k+1}=aD_k-2D_{k-1}+E_k\quad(1\le k<N),
\]
where $|E_0|\le(a+1)S_0+2K_0$ and $|E_k|\le(a+1)S_k+2S_{k-1}+2K_0$. The constant $K_0=\ell C+\max_i|\log|d(\lambda_i)||$ depends only on the common background parameters. In particular, $D_{-1}=0$ incorporates the initialized transition into the same recurrence.
\end{lemma}
\begin{proof}\hspace{0.5em}
By Lemma~\ref{lem:twobound}, $\log|p_k(\lambda_i)|$ differs from $2(M_{L_k,k}-P_{i,k})$ by at most $C$ on $K_k$ and is bounded by $C$ in absolute value on $L_k$. For the log amplitudes $P_{i,k}=\tfrac12\log u_{i,k}$, \eqref{eq:B} reads
\[
 P_{i,k+1}=P_{i,k}+(\ell-1)\log|p_k(\lambda_i)|
                         +\log|p_{k-1}(\lambda_i)|.
\]
If $i\in K_k$, it belongs to $L_{k-1}$, so $\log|p_{k-1}(\lambda_i)|$ is bounded. For $i\in L_k$, it belonged to $K_{k-1}$, while $\log|p_k(\lambda_i)|$ is bounded. Using $1-2(\ell-1)=-a$ and $2(\ell-1)=a+1$ gives, for $k\ge1$,
\begin{align}
 P_{i,k+1}&=-aP_{i,k}+(a+1)M_{L_k,k}+r_{i,k},&&i\in K_k,
 \label{eq:top}\\
 P_{i,k+1}&=P_{i,k}-2P_{i,k-1}+2M_{L_{k-1},k-1}+r_{i,k},
 &&i\in L_k,\label{eq:tail}
\end{align}
where $|r_{i,k}|\le K_0$, with $K_0=\ell C+\max_i|\log|d(\lambda_i)||$. This bounds both recurrent and initialization errors. Put $e_k=M_{L_k,k}-U_k$, so $0\le e_k\le S_k$. Since $L_k=K_{k-1}=K_{k+1}$ and $K_k=L_{k+1}$, taking means gives
\[
 \begin{aligned}
 T_{k+1}&=U_k-2T_{k-1}+2U_{k-1}+2e_{k-1}+\eta^T_k,\\
 U_{k+1}&=-aT_k+(a+1)U_k+(a+1)e_k+\eta^U_k,
 \end{aligned}
 \qquad |\eta^T_k|,|\eta^U_k|\le K_0.
\]
Their difference has forcing $E_k=2e_{k-1}-(a+1)e_k+\eta^T_k-\eta^U_k$. At initialization the component update is
\[
 P_{i,1}=P_{i,0}+\log|d(\lambda_i)|
                      +(\ell-1)\log|p_0(\lambda_i)|.
\]
Thus $T_1=U_0+\eta^T_0$ and $U_1=-aT_0+(a+1)(U_0+e_0)+\eta^U_0$, with the same error bound. These identities give
\begin{align}
 D_{k+1}&=aD_k-2D_{k-1}+E_k,
 &|E_k|&\le(a+1)S_k+2S_{k-1}+2K_0,\label{eq:Drec}\\
 D_1&=aD_0+E_0,
 &|E_0|&\le(a+1)S_0+2K_0.\label{eq:Dinit}
\end{align}
In particular, the second initial state is imposed by the actual initialization rather than chosen independently.

\end{proof}

\begin{lemma}[Spread and forcing bounds]\label{lem:rollingSpread}
Under the same finite-segment separation assumption, put $R_0=S_0+1$. There are $C_s\ge1$ and $C_e>0$, independent of $N$ and the initial gradient, such that
\[
 S_k\le C_sR_0\nu^k\quad(0\le k\le N),\qquad
 |E_k|\le C_eR_0\nu^k\quad(0\le k<N).
\]
Here $E_k$ is the gap forcing in Lemma~\ref{lem:rollingRecurrences}; both constants depend only on the common background parameters.
\end{lemma}
\begin{proof}\hspace{0.5em}
We subtract the group means and obtain a centered error $\xi_{i,k}$ with $|\xi_{i,k}|\le2K_0$. For $i\in K_k$, $\delta_{i,k+1}=-a\delta_{i,k}+\xi_{i,k}$. For $i\in L_k=K_{k-1}$, $\delta_{i,k+1}=\delta_{i,k}-2\delta_{i,k-1}+\xi_{i,k}$. The previous relation for the dominant group gives $\delta_{i,k}=-a\delta_{i,k-1}+\xi_{i,k-1}$. Substitution gives
\begin{equation}\label{eq:spread}
 \delta_{i,k+1}=-B\delta_{i,k-1}+O(1).
\end{equation}
The error here is $\xi_{i,k}+\xi_{i,k-1}$, of modulus at most $4K_0$. Each group is dominant every second cycle, so the centered deviations grow at most on the scale $B^{k/2}=\nu^k$. Including the first transition and the intervening cycles gives
\begin{equation}\label{eq:Sbound}
 S_k\le C_s(S_0+1)\nu^k.
\end{equation}
We verify these estimates with the constants in the common background.

At each later cycle, the component error is the sum of $\ell-1$ errors from the current polynomial and one from the previous polynomial, so its modulus is at most $\ell C$. The initialized error is at most $(\ell-1)C+\max_i|\log|d(\lambda_i)||$. Both are bounded by $K_0$.

Subtracting a group mean from an error bounded by $K_0$ gives a centered error of modulus at most $2K_0$. The two-step relation \eqref{eq:spread} adds two such centered errors, giving $4K_0$; there is no factor $a$ multiplying these errors. The first tail-group spread is at most $S_0+2K_0$. For a group, $k_0$ is the first cycle where it is dominant, so $k_0=0$ for $I$ and $k_0=1$ for $J$. With $R_0=S_0+1$, the dominant parity subsequences therefore satisfy
\[
 |\delta_{i,k_0+2j}|
 \le B^j(S_0+2K_0)+4K_0\frac{B^j-1}{B-1}
 \le C_0 R_0 B^j,\qquad k_0\in\{0,1\}.
\]
The additive $1$ in $C_0$ controls $S_0$, $2K_0$ covers the initialized tail, and $4K_0/(B-1)$ is the two-step geometric sum. If $k\ge0$ is a dominant time, the intervening spread at $k+1$ is at most $a$ times that spread plus $2K_0$. With $B=\nu^2$, $\nu>1$, and $R_0\nu^k\ge1$, this gives the uniform estimate
\[
 aC_0R_0\nu^k+2K_0
 \le(aC_0+2K_0)R_0\nu^k
 \le C_sR_0\nu^{k+1}.
\]
Thus $C_s=(a+1)C_0+2K_0$ covers both dominant parity subsequences and their intervening times in \eqref{eq:Sbound}. The gap estimate then gives, for $k\ge1$,
\[
 |E_k|\le(a+1)S_k+2S_{k-1}+2K_0
 \le[(a+3)C_s+2K_0]R_0\nu^k .
\]
At initialization, \eqref{eq:Dinit} gives $|E_0|\le(a+1)S_0+2K_0\le C_eR_0$. This proves the asserted forcing bound, including initialization.
\end{proof}

\begin{lemma}[A forced scalar recurrence]\label{lem:forcedRecurrence}
Let $a>3$, $\sigma=(a+\sqrt{a^2-8})/2$, $\sigma_-=2/\sigma$, and $1<\nu<\sigma$. Suppose $X_{-1}=0$, $X_0\in\mathbb R$, and $X_{k+1}=aX_k-2X_{k-1}+f_k$ for $0\le k<N$, with $|f_k|\le E\nu^k$ and $E\ge0$. Then, for $0\le k\le N$,
\[
 |X_k-X_0\varphi_k|\le \frac{C_fE}{\sigma-\nu}\sigma^k,
 \qquad \sigma^k\le\varphi_k\le C_f\sigma^k,
\]
where $\varphi_k=(\sigma^{k+1}-\sigma_-^{k+1})/(\sigma-\sigma_-)$ and $C_f=\sigma/(\sigma-\sigma_-)$. These constants are independent of $N$, $X_0$, and the forcing sequence.
\end{lemma}
\begin{proof}\hspace{0.5em}
The fundamental solution
\[
 \varphi_k=\frac{\sigma^{k+1}-\sigma_-^{k+1}}{\sigma-\sigma_-}
 \quad\text{satisfies}\quad \varphi_0=1,\ \varphi_1=a,
 \quad\sigma^k\le \varphi_k\le C_f\sigma^k.
\]
Here $C_f=\sigma/(\sigma-\sigma_-)$. To check the lower bound, put $t=\sigma_-/\sigma\in(0,1)$ and note that $\varphi_k/\sigma^k=(1-t^{k+1})/(1-t)\ge1$; dropping $t^{k+1}$ gives the upper bound. The fundamental solution depends only on $a$.

If $X_{-1}=0$, $X_0$ is given, and $X_{k+1}=aX_k-2X_{k-1}+f_k$, variation of constants gives
\begin{equation}\label{eq:commonConvolution}
 X_k=X_0\varphi_k+\sum_{j=0}^{k-1}f_j\varphi_{k-1-j},\qquad
 |X_k-X_0\varphi_k|\le\frac{C_f E}{\sigma-\nu}\sigma^k
 \quad\text{if }|f_j|\le E\nu^j.
\end{equation}
We can check the formula by substitution, using $\varphi_{-1}=0$ and $\varphi_0=1$. The bound follows by summing $\sum_{j<k}\nu^j\sigma^{k-1-j}\le\sigma^k/(\sigma-\nu)$.

\end{proof}

For the gap in Lemma~\ref{lem:rollingRecurrences}, the variation-of-constants formula is
\begin{equation}\label{eq:Dvariation}
 D_k=D_0\varphi_k+\sum_{j=0}^{k-1}E_j\varphi_{k-1-j}.
\end{equation}
Lemma~\ref{lem:rollingSpread} bounds the forcing by $C_eR_0\nu^k$, so Lemma~\ref{lem:forcedRecurrence} bounds the sum by $C_dR_0\sigma^k$, with $C_d=C_eC_f/(\sigma-\nu)$. This is the estimate used in the main-text comparison lemma.

Take $C_s,C_d$ from Lemma~\ref{lem:rollingComparison}. Replacing $R_0$ by $s+1$ gives the explicit sufficient threshold
\begin{equation}\label{eq:explicitD}
 D_*(s)=1+\max\{2C_d(s+1),\ 8C_s(s+1),\ 4\Gamma_*\}.
\end{equation}
Together with Lemma~\ref{lem:rollingComparison}, it gives the finite-time bounds \eqref{eq:Dcontrol} and \eqref{eq:invariance} from the first stored gradient onward.

\subsection{Recurrences for overlapping groups}\label{app:overlapRecurrences}

Here $n=m+r$ with $2\le r<m$, and the ordered groups are $L=\{1,\ldots,r\}$, $C=\{r+1,\ldots,m\}$, and $R=\{m+1,\ldots,m+r\}$. Keep the constant-step initialization and scales from Appendix~\ref{app:rollingRecurrences}, with a full-support initial gradient. Choose $\Gamma_*,C$ uniformly from Lemma~\ref{lem:twobound} for these two overlapping sets, and define $K_0=\ell C+\max_i|\log|d(\lambda_i)||$. Use the formulas for $C_s,C_e,C_f,C_d$ in Appendix~\ref{app:rollingRecurrences} with this choice of $C$. Let $P_{i,k}=\log|h_{i,k}|$. Let $T_k$ be the mean over $L$ at even cycles and over $R$ at odd cycles, and $U_k$ the other outer mean. Define $D_k=T_k-U_k$, $S_k$ as the largest deviation from either outer mean, and $Y_{i,k}=P_{i,k}-(T_k+U_k)/2$ for $i\in C$. On a finite segment $0\le k<N$, assume that $L\cup C$ at even cycles and $C\cup R$ at odd cycles are separated from their complements by at least $\Gamma_*$. All constants below are uniform in the segment length and the initial gradient.

\begin{lemma}[Forcing of the common offsets]\label{lem:commonOffset}
On the separated segment just described, the offsets satisfy $Y_{i,k+1}=-aY_{i,k}-2Y_{i,k-1}+F_{i,k}$ for $1\le k<N$, with
\[
 |F_{i,k}|\le\tfrac{a+1}{2}S_k+S_{k-1}+2K_0.
\]
The initialized update is $Y_{i,1}=-aY_{i,0}+F_{i,0}$, where $|F_{i,0}|\le(a+1)S_0/2+2K_0$. If $S_k\le C_s(S_0+1)\nu^k$, both bounds imply $|F_{i,k}|\le C_e(S_0+1)\nu^k$ with $C_e=(a+3)C_s+2K_0$.
\end{lemma}
\begin{proof}\hspace{0.5em}
For Theorem~\ref{thm:overlap}, the outer-gap and common-component recurrences share the forcing constant $C_e$ and convolution constant $C_d$. The calculation below checks both the later updates and the actual initialization, using the constants in the one-sided polynomial bound for the two overlapping sets. Let $M_k$ be the largest log amplitude in the outer group with smaller components and set $e_k=M_k-U_k$, so $0\le e_k\le S_k$. The quantity $M_{k-1}$ is the maximum in the outer group with smaller components at cycle $k-1$, which becomes the dominant outer group at cycle $k$. For $k\ge1$, write the three updates with their bounded errors as
\[
\begin{aligned}
 T_{k+1}&=U_k-2T_{k-1}+2M_{k-1}+\eta^T_k,\\
 U_{k+1}&=-aT_k+(a+1)M_k+\eta^U_k,\\
 P_{i,k+1}&=-aP_{i,k}+(a+1)M_k-2P_{i,k-1}
                  +2M_{k-1}+\eta_{i,k}\quad(i\in C).
\end{aligned}
\]
The bound on the polynomial values gives $|\eta^T_k|,|\eta^U_k|,|\eta_{i,k}|\le K_0$. Subtracting $(T_{k+1}+U_{k+1})/2$ from the common update now gives
\[
 F_{i,k}=\frac{a+1}{2}e_k+e_{k-1}+\eta_{i,k}
                    -\frac{\eta^T_k+\eta^U_k}{2}.
\]
Thus $|F_{i,k}|\le(a+1)S_k/2+S_{k-1}+2K_0$. At initialization, the actual updates are
\[
\begin{gathered}
 T_1=U_0+\eta^T_0,\qquad
 U_1=-aT_0+(a+1)M_0+\eta^U_0,\\
 P_{i,1}=-aP_{i,0}+(a+1)M_0+\eta_{i,0}.
\end{gathered}
\]
The initialization errors also have absolute value at most $K_0$. The same subtraction gives $F_{i,0}=(a+1)e_0/2+\eta_{i,0}-(\eta^T_0+\eta^U_0)/2$, so $|F_{i,0}|\le(a+1)S_0/2+2K_0$. Using $S_k\le C_s(S_0+1)\nu^k$, both bounds imply $|F_{i,k}|\le C_e(S_0+1)\nu^k$ with the same $C_e=(a+3)C_s+2K_0$ as for the outer gap.
\end{proof}

\begin{lemma}[Outer and common-component recurrences]\label{lem:overlapDynamics}
On the same separated segment, put $R_0=S_0+1$ and use the virtual states $D_{-1}=Y_{i,-1}=0$. Then
\[
 D_{k+1}=aD_k-2D_{k-1}+E_k,\qquad
 Y_{i,k+1}=-aY_{i,k}-2Y_{i,k-1}+F_{i,k}
\]
for $0\le k<N$, with $|E_k|,|F_{i,k}|\le C_eR_0\nu^k$. The spread satisfies $S_k\le C_sR_0\nu^k$ for $0\le k\le N$.
\end{lemma}
\begin{proof}\hspace{0.5em}
The one-sided estimate proved in Appendix~\ref{app:constants} applies to both $I=L\cup C$ and $I=C\cup R$. For their respective complements $J$, if $\min_I P-\max_J P\ge\Gamma_*$, then
\begin{equation}\label{eq:oneSidedFilter}
 \log|p(\lambda_i)|=
 \begin{cases}2(\max_J P-P_i)+O(1),&i\in I,\\O(1),&i\in J.
 \end{cases}
\end{equation}
The bounds are uniform in the ratios within the two sets.

On a finite segment with the proposed ordering, $T_k$ and $U_k$ are the means of the dominant outer group and the other outer group. For even $k\ge2$, their labels over three consecutive cycles are
\[
 \begin{array}{c|ccc}
 \text{mean of}&k-1&k&k+1\\ \hline
 L&U_{k-1}&T_k&U_{k+1}\\
 R&T_{k-1}&U_k&T_{k+1}
 \end{array}
\]
For odd $k$, interchange $L$ and $R$ in this display. The spread $S_k$ concerns only the outer groups; $Y_{i,k}$ tracks each common component relative to their midpoint. For $p_k$, we use the $m$-component group $L\cup C$ at even $k$ and $C\cup R$ at odd $k$ in the polynomial bound. The dominant outer group belongs to this group at cycle $k$ and to its complement at cycle $k-1$. The other outer group has the reverse roles. Every common component belongs to both $m$-component groups. Thus the outer components satisfy \eqref{eq:top}--\eqref{eq:tail}, with the maximum over the other outer group taken over the opposite outer group. Their bounds do not involve $Y_{i,k}$, since \eqref{eq:oneSidedFilter} is uniform within the group used in the bound. Thus, with $a=2\ell-3$ and $R_0=S_0+1$,
\begin{align}
 S_k&\le C_sR_0\nu^k,\qquad
 D_{k+1}=aD_k-2D_{k-1}+E_k,\quad |E_k|\le C_eR_0\nu^k,
 \label{eq:overlapOuter}
\end{align}
where $D_{-1}=0$ incorporates the bounded initialization error.

Write $M_k=U_k+e_k$ for the largest log amplitude in the outer group with smaller components, where $0\le e_k\le S_k$. A common component is in the dominant set for both $p_k$ and $p_{k-1}$. Subtracting the outer midpoint from its update gives
\begin{equation}\label{eq:overlapCore}
 Y_{i,k+1}=-aY_{i,k}-2Y_{i,k-1}+F_{i,k},\qquad
 |F_{i,k}|\le \tfrac{a+1}{2}S_k+S_{k-1}+2K_0 \quad(k\ge1).
\end{equation}
The first transition gives $Y_{i,1}=-aY_{i,0}+F_{i,0}$ with $|F_{i,0}|\le \tfrac{a+1}{2}S_0+2K_0$, so the virtual state is $Y_{i,-1}=0$. With the spread bound in \eqref{eq:overlapOuter}, these estimates give $|F_{i,k}|\le C_eR_0\nu^k$ for all $k\ge0$. Lemma~\ref{lem:commonOffset} gives the midpoint calculation and its initialized transition.

\end{proof}

An explicit sufficient choice for \eqref{eq:overlapInitial} is
\begin{equation}\label{eq:overlapThreshold}
 D_*(s)=1+\max\{2C_d(s+1),16C_s(s+1),
                16[C_fs+C_d(s+1)],8\Gamma_*\}.
\end{equation}
\section{Auxiliary estimates for the sampled method}\label{app:sampling}

\subsection{Jacobian and inverse branches}

Here $m\ge1$, $n>m$, $\ell\ge2$, and $0<\lambda_1<\cdots<\lambda_n$.

At SCR cycle $k$, start with $u_i=(F_k)_i^2$. Fix a reference index $r$ and divide all weights by $u_r$, so that $u_r=1$. Let $z_i=\log(u_i/u_r)$ for $i\ne r$, and write $d=n-1$ for the dimension of this coordinate space. For the moment omit the sampling polynomial and consider the map given by $m\ell$ Ritz updates, with each current reciprocal Ritz value used $\ell$ times:
\begin{equation}\label{eq:T}
 T(z)_i=z_i+2\ell\log\left|\frac{\pi(\lambda_i;z)}
 {\pi(\lambda_r;z)}\right|.
\end{equation}
For $c=(c_0,\ldots,c_{m-1})\in\R^m$, write $\pi_c(\xi)=\xi^m+\sum_{a=0}^{m-1}c_a\xi^a$. The coefficient map $R(z)$ assigns to $z$ the coefficients of the monic polynomial satisfying \eqref{eq:orthogonal} for $u_i=e^{z_i}$, $i\ne r$, and $u_r=1$. Define $L(c)=(2\log|\pi_c(\lambda_i)/\pi_c(\lambda_r)|)_{i\ne r}$. Then $T(z)=z+\ell L(R(z))$. All maps are restricted to the domain where the required polynomial values are nonzero. We also use $v_i=(1,\lambda_i,\ldots,\lambda_i^{m-1})^T$. The symbol $D$ before a map denotes its Jacobian matrix; $\vol$ denotes Lebesgue measure in these log coordinates.

\begin{lemma}[Multiple-Ritz Jacobian]\label{lem:jacA}
On the open set where no polynomial value vanishes,
\begin{equation}\label{eq:DRDL}
 DR(z)DL(R(z))=-2I_m,
 \qquad \det DT(z)=(1-2\ell)^m.
\end{equation}
\end{lemma}
\begin{proof}\hspace{0.5em}
Put $H=\sum_i u_iv_iv_i^T\succ0$. Differentiating \eqref{eq:orthogonal} gives, for $i\ne r$,
\[
 \frac{\partial R}{\partial z_i}
 =-u_i\pi(\lambda_i)H^{-1}v_i,
 \qquad
 \frac{\partial L_i}{\partial c}
 =2\left(\frac{v_i}{\pi(\lambda_i)}
 -\frac{v_r}{\pi(\lambda_r)}\right)^T.
\]
Writing $\pi_i=\pi(\lambda_i)$, multiplication gives
\[
 DR\,DL=-2H^{-1}\left[
 \sum_{i\ne r}u_iv_iv_i^T
 -\frac{\sum_{i\ne r}u_i\pi_i v_i}{\pi_r}\,v_r^T\right].
\]
Equation~\eqref{eq:orthogonal} gives $\sum_{i\ne r}u_i\pi_i v_i=-u_r\pi_r v_r$, so the bracket is exactly $H$. This proves $DR\,DL=-2I_m$. Here $DR$ has size $m\times d$ and $DL$ has size $d\times m$. Since $DT=I_d+\ell DL\,DR$, Sylvester's determinant identity gives $\det(I_d+\ell DL\,DR)=\det(I_m+\ell DR\,DL)=(1-2\ell)^m$.
\end{proof}

\begin{lemma}[A unique polynomial for each root-gap pattern]\label{lem:inversePolynomial}
Let $\ell\ge2$ and fix positive numbers $U_1,\ldots,U_n$. For every choice of $m$ distinct open spectral gaps, there is exactly one monic degree-$m$ polynomial $\pi_c$ with a root in each chosen gap satisfying
\begin{equation}\label{eq:inversecondition}
 \sum_i\frac{U_iv_i}{\pi_c(\lambda_i)^{2\ell-1}}=0.
\end{equation}
Here $v_i=(1,\lambda_i,\ldots,\lambda_i^{m-1})^T$.
\end{lemma}
\begin{proof}\hspace{0.5em}
Fix $m$ different gaps. Choose one root in each to obtain a monic polynomial, and record its strict signs at the support points. The region in coefficient space consists of monic polynomials with those signs. It is nonempty, open and convex, since each sign condition is a strict linear inequality in the coefficients. Conversely, the intermediate value theorem gives a root in each of the $m$ chosen gaps. Degree $m$ leaves exactly one root in each, with no other real or complex roots. All roots lie in $[\lambda_1,\lambda_n]$, so their elementary symmetric functions, and hence all coefficients in this region, are bounded.

To solve \eqref{eq:inversecondition} in this region, consider
\[
 \Phi(c)=\sum_i\frac{U_i}{\pi_c(\lambda_i)^{2\ell-2}}.
\]
At any boundary point of the bounded region, some strict sign condition becomes an equality, so a denominator vanishes and $\Phi$ tends to infinity. A sublevel set containing one interior point is therefore a compact subset of the region, so $\Phi$ has a minimum there. Its Hessian is
\[
 (2\ell-2)(2\ell-1)
 \sum_i\frac{U_iv_iv_i^T}{\pi_c(\lambda_i)^{2\ell}}.
\]
This is positive definite: the coefficients are positive, and the vectors $v_i$ span $\mathbb R^m$ by the Vandermonde rank condition. Strict convexity gives exactly one minimizer in the region. Differentiating $\Phi$ shows that its stationarity condition is \eqref{eq:inversecondition}.

\end{proof}

\begin{lemma}[Exact inverse-branch count]\label{lem:branches}
For $\ell\ge2$, every target $z'\in\R^{n-1}$, representing positive weight ratios, has exactly $N=\binom{n-1}{m}$ inverse images in the domain of $T$ where $\pi(\lambda_i)\ne0$ for every $i$. There is one in each choice of $m$ different spectral gaps for the roots of the source polynomial. Consequently, with $\rho=N/(2\ell-1)^m$,
\begin{equation}\label{eq:volumeA}
 \vol(T^{-1}E)=\rho\vol(E)
\end{equation}
for every measurable set $E\subseteq\R^{n-1}$.
\end{lemma}
\begin{proof}\hspace{0.5em}
If two roots $\alpha,\beta$ lay in the same open spectral gap, put $s(\xi)=\pi(\xi)/[(\xi-\alpha)(\xi-\beta)]$. Every support point lies outside that gap, so
\[
 \sum_i u_i\pi(\lambda_i)s(\lambda_i)
 =\sum_i u_i(\lambda_i-\alpha)(\lambda_i-\beta)s(\lambda_i)^2>0.
\]
At least one summand is positive because $\deg s=m-2<n$. By linearity, \eqref{eq:orthogonal} requires this sum to be zero for every polynomial of degree less than $m$, a contradiction. Thus a polynomial that is nonzero at every support point has at most one root per open gap, and there are at most $N$ root-gap patterns.

Let $U_i>0$ represent the target weights. Lemma~\ref{lem:inversePolynomial} gives exactly one polynomial satisfying \eqref{eq:inversecondition} for each choice of $m$ distinct gaps. For this polynomial, define the normalized positive weights
\[
 u_i=\frac{U_i\,\pi_c(\lambda_r)^{2\ell}}
           {U_r\,\pi_c(\lambda_i)^{2\ell}},
 \qquad u_r=1.
\]
Equation~\eqref{eq:inversecondition} gives $\sum_i u_i\pi_c(\lambda_i)v_i=0$. Thus $\pi_c$ is the monic polynomial satisfying \eqref{eq:orthogonal} for these weights and their image under $T$ has ratios $U_i/U_r$. Conversely, every inverse with nonzero polynomial values at the support points gives these weights and a stationary point in its root-gap region, so it is the inverse already constructed. Different patterns cannot give the same source point, since its monic polynomial satisfying \eqref{eq:orthogonal} is unique. At a point where all polynomial values on the support are nonzero, Lemma~\ref{lem:history} gives $m$ distinct interior roots; the nonzero-value condition also rules out roots at support points. The gap argument above places them in $m$ different gaps, so every such point belongs to exactly one pattern. On each pattern $T$ is therefore a bijection onto $\mathbb R^d$; its local inverses from Lemma~\ref{lem:jacA} agree by uniqueness and give one smooth global inverse branch. Each branch has absolute Jacobian $(2\ell-1)^{-m}$. Adding the change-of-variables formulas over the $N$ disjoint branches proves \eqref{eq:volumeA}.
\end{proof}

The sampling translation in \eqref{eq:translated} is independent of the log ratios, so it preserves the Jacobian and inverse-image count just proved.

\subsection{Sublevel estimates for near cancellation}\label{app:samplingCancellation}

For CSD the near-cancellation estimate in \citet{GuCSD} applies to a sum of single exponentials and a nonzero constant. The subset expansion \eqref{eq:subset} for several Ritz values also contains products of distinct weights. The following estimate handles those terms.

\begin{lemma}[Sublevel measure of an exponential polynomial]\label{lem:thin}
Let $d\ge1$ be an integer and
\[
 F(y)=b+\sum_{\varnothing\ne J\subseteq\{1,\ldots,d\}}
 a_J\exp\left(\sum_{j\in J}y_j\right),\qquad b\ne0.
\]
There are constants $C$ and $\epsilon_0>0$, depending only on $d$ and $b$, such that on every axis-parallel box of side length $L\ge1$,
\begin{equation}\label{eq:thin}
 \vol_d\{y\text{ in the box}:|F(y)|<\epsilon\}
 \le C\epsilon^{1/d}L^{d-1}\quad(0<\epsilon<\epsilon_0).
\end{equation}
The constants do not depend on the other coefficients or on the position of the box.
\end{lemma}
\begin{proof}\hspace{0.5em}
Divide $F$ and $\epsilon$ by $b$ and $|b|$, respectively. For $d=1$, the set is empty when the exponential coefficient is zero. Otherwise the reverse triangle inequality places it in an interval of length at most $\log[(1+\epsilon)/(1-\epsilon)]\le4\epsilon$ for $\epsilon<1/2$.

For $d>1$, write $F=A(y')+e^{y_d}B(y')$, where $A$ has constant term one and the same square-free exponential form in $d-1$ variables. On the $(d-1)$-dimensional box of side $L$, the induction hypothesis bounds the set $|A|<\eta$ by $C\eta^{1/(d-1)}L^{d-2}$. Integrating over the remaining interval of length $L$ gives $C\eta^{1/(d-1)}L^{d-1}$. On each remaining slice $|A|\ge\eta$. If $B=0$, the sublevel set is empty for $\epsilon<\eta$. Otherwise $\big||B|e^{y_d}-|A|\big|<\epsilon$ is necessary, so its length is at most
\[
 \log\frac{|A|+\epsilon}{|A|-\epsilon}
 \le\frac{4\epsilon}{|A|}
 \le\frac{4\epsilon}{\eta}\qquad(\epsilon<\eta/2).
\]
These slices contribute at most $4\epsilon L^{d-1}/\eta$. Choose $\eta=\epsilon^{(d-1)/d}$, so that $\eta^{1/(d-1)}=\epsilon/\eta=\epsilon^{1/d}$. Both terms are then $O(\epsilon^{1/d}L^{d-1})$. A sufficiently small $\epsilon_0$ ensures the induction hypotheses and $\epsilon<\eta/2$. Restoring $b$ changes only the constants depending on $d,b$.
\end{proof}

Fix $m\ge1$, $n>m$, and a reference $r$, and write $d=n-1$. Normalize the positive weights by $u_r=1$ and set $z_i=\log(u_i/u_r)$ for $i\ne r$. Rescaling the support to $[0,1]$ makes the nonzero values of the monic polynomial smaller than one in modulus, so their negative logarithms are nonnegative. Put $\Delta_\lambda=\lambda_n-\lambda_1$ and $\beta_i=(\lambda_i-\lambda_1)/\Delta_\lambda$, so that $0=\beta_1<\cdots<\beta_n=1$. Write $\widehat\pi$ for the monic polynomial satisfying \eqref{eq:orthogonal} with $\lambda_i$ replaced by $\beta_i$ and with the same weights; then $\pi(\lambda_i)=\Delta_\lambda^m\widehat\pi(\beta_i)$. Formula~\eqref{eq:subset} holds with these nodes. Let $\mathcal N_i(u)$ denote its numerator evaluated at $\beta_i$; only subsets not containing $i$ occur. Each $\mathcal N_i$ is nonzero: every coefficient indexed by a subset not containing $i$ is nonzero, and different subsets give different square-free monomials. After $u_r=1$, subsets containing $r$ give monomials of degree $m-1$, while the others give degree $m$. Within either class, different subsets leave different sets of variables. Hence no monomials merge, and the normalized polynomial remains nonzero.

To distinguish near-cancellation from a small normalizing monomial, for each $i\ne r$ and each $m$-element subset $S_0$ containing $r$ but not $i$, put $F_{i,S_0}(z)=\mathcal N_i(u)/u_{S_0},\qquad u_r=1$. This choice makes the $S_0$ term a nonzero constant. 

\begin{lemma}[Normalized numerator sublevel bounds]\label{lem:normalizedSublevel}
For the finitely many functions $F_{i,S_0}$ defined above and $d>1$, there are $C,\epsilon_0>0$, depending only on the spectrum and $m$, such that, on every axis-parallel box $\mathcal B$ of side $L\ge1$,
\begin{equation}\label{eq:thinp}
 \vol_d\{z\in\mathcal B:|F_{i,S_0}(z)|<\epsilon\}
 \le C\epsilon^{1/(d-1)}L^{d-1}\qquad(0<\epsilon<\epsilon_0).
\end{equation}
For $d=1$, these functions are nonzero constants.
\end{lemma}
\begin{proof}\hspace{0.5em}
Reversing the signs of the coordinates $z_j$ for $j\in S_0\setminus\{r\}$ makes all variable exponents either zero or one. Also $F_{i,S_0}$ is independent of $z_i$. For $n\ge3$ and $d=n-1$, apply Lemma~\ref{lem:thin} in the remaining $d-1$ variables, then multiply by the length $L$ of the $z_i$ interval. This proves \eqref{eq:thinp}. Only finitely many such polynomials occur. For $n=2$ they are nonzero constants and no sublevel estimate is needed.
\end{proof}

\subsection{Small-decrease events and first entry}\label{app:samplingEvents}

Throughout this subsection, $m\ge1$, $n>m$, $\ell\ge2$, $d=n-1$, and the fixed positive spectrum is normalized by $\beta_i=(\lambda_i-\lambda_1)/(\lambda_n-\lambda_1)$. Fix a reference $r$. For positive weights $u$ put $u_r=1$ and $z_i=\log(u_i/u_r)$. Let $\widehat\pi$ be the monic polynomial satisfying \eqref{eq:orthogonal} with $\lambda_i$ replaced by $\beta_i$, and use the normalized numerators $F_{i,S_0}$ from the preceding subsection. Consider any family of prescribed sampling sequences satisfying \eqref{eq:samplingbound} with common constants. For every such sequence,
\[
 z_{i,k+1}-z_{i,k}=2\ell(a_{r,k}-a_{i,k})+(\Delta_k)_i,
 \quad a_{i,k}=-\log|\widehat\pi_k(\beta_i)|\ge0,
 \quad \|\Delta_k\|_\infty\le D_\Delta.
\]
We restrict to trajectories on which these values are nonzero. The excluded finite-time zero sets are null. Write $S_{r,N}=\sum_{k<N}a_{r,k}$ and $\rho=\binom{n-1}{m}/(2\ell-1)^m<1$. Each prescribed length-$k$ composition $\Psi_w$ satisfies $\vol(\Psi_w^{-1}E)=\rho^k\vol(E)$ by \eqref{eq:nonautonomous}. The parameters $R\ge1$, $M_0,\kappa,b>0$ below are fixed independently of the sampling sequence.

\begin{lemma}[Weights controlled by the reference polynomial value]\label{lem:referenceWeights}
For positive weights normalized by $u_r=1$, let $t_1\ge\cdots\ge t_d$ be the other weights in decreasing order. There is $C\ge1$, depending only on the normalized nodes and $m$, such that $|\widehat\pi(\beta_r)|\le Ct_m$. Consequently, if $a_r=-\log|\widehat\pi(\beta_r)|$ is finite, at least $m$ other weights are at least $C^{-1}e^{-a_r}$.
\end{lemma}
\begin{proof}\hspace{0.5em}
If the other weights in decreasing order are $t_1\ge\cdots\ge t_d$, the subset formula gives
\begin{equation}\label{eq:referenceweight}
 |\widehat\pi(\beta_r)|
 \le C\frac{t_1\cdots t_m}{t_1\cdots t_{m-1}}=Ct_m.
\end{equation}
For the numerator, every product is at most $t_1\cdots t_m$, and the sum of the absolute coefficients depends only on the normalized nodes. For the denominator, keep $r$ and the largest $m-1$ other weights; its coefficient is at least $\min_{|S|=m}\Delta_S^2>0$ at the normalized nodes. Take $C\ge1$ large enough for \eqref{eq:referenceweight} and for the sum of the coefficients of $Z_m$. Thus $|\widehat\pi(\beta_r)|=e^{-a_r}$ implies $t_m\ge C^{-1}e^{-a_r}$.
\end{proof}

\begin{lemma}[Uniform bounds up to the first near cancellation]\label{lem:samplingBoxes}
Fix $r$, $R\ge1$, $M_0>0$, $\kappa>0$, and $b>0$. Let $D_\Delta$ bound $\|\Delta_k\|_\infty$ for the sampling sequences under consideration. Use the log ratios $z_k$ and the nonnegative quantities $a_{i,k}$ from the common background of this subsection. Set $W_N=M_0e^{\kappa N}$ for $N\ge1$. There exist $U_N,V_N,A_N>0$, a box $\mathcal B_N=[-V_N,U_N]^d$ containing $[-R,R]^d$, and a set $\mathcal C_N\subseteq\mathcal B_N$, all independent of the particular sampling sequence, with the following properties.

Suppose the log-ratio trajectory is defined through time $N$, $z_0\in[-R,R]^d$, and $S_{r,N}\le W_N$. Then $z_{i,k}\le U_N$ for $0\le k\le N$. Under the same assumptions, at any $k<N$ with $z_k\in\mathcal B_N\setminus\mathcal C_N$, one has $a_{i,k}\le A_N$ for $i\ne r$. The lower bound for the box is given by $V_N=R+N(2\ell A_N+D_\Delta)$. For $d>1$ the volume bounds are
\[
 \vol(\mathcal B_N)\le C N^d e^{\kappa dN},\qquad
 \vol(\mathcal C_N)\le C N^{d-1}
       e^{[\kappa(d-1)-b/(d-1)]N}.
\]
For $d=1$ the bound for $\mathcal B_N$ still holds, and $\mathcal C_N$ is empty for all sufficiently large $N$.
\end{lemma}
\begin{proof}\hspace{0.5em}
For $z_0\in[-R,R]^d$ and $S_{r,N}\le W_N$, we sum the ratio recurrence and drop the nonpositive term $-a_{i,k}\le0$ to get $z_{i,k}\le R+2\ell\sum_{j<k}a_{r,j}+kD_\Delta$. Since the reference logarithmic quantities are nonnegative, for $0\le k\le N$ this gives
\begin{equation}\label{eq:upperz}
 z_{i,k}\le R+2\ell W_N+ND_\Delta=:U_N.
\end{equation}
This bounds the log ratios from above when $S_{r,N}\le W_N$. To obtain a lower bound, we use the fact that a small $S_{r,N}$ prevents too many of the other weights from being small at the same time. By Lemma~\ref{lem:referenceWeights}, $a_{r,k}\le W_N$ ensures that at least $m$ other weights are at least $C^{-1}e^{-W_N}$. Choose $C\ge1$ to cover this bound and the sum of the coefficients of $Z_m$ at the normalized nodes.

Excluding $i$, if present, leaves at least $m-1$ of these weights. Choose $m-1$ of them and adjoin $r$ to obtain $S_0$, also when $m=1$ and the chosen group is empty. This choice gives the following lower bound for the normalizing monomial: $u_{S_0}\ge C^{-(m-1)}e^{-(m-1)W_N},\qquad Z_m(u)\le Ce^{mU_N}.$ Put $\epsilon_N=e^{-bN}$. At the normalized nodes, $|\widehat\pi_k(\beta_i)|=|F_{i,S_0}(z_k)|u_{S_0}/Z_m(u)$. We use this identity and the bounds for the monomial and denominator to obtain, when $|F_{i,S_0}(z_k)|\ge\epsilon_N$,
\begin{equation}\label{eq:lowerfilter}
 \log|\widehat\pi_k(\beta_i)|
 \ge-bN-(m-1)W_N-mU_N-m\log C=:-A_N.
\end{equation}
This is the upper bound $a_{i,k}\le A_N$. In the ratio update the reference contribution is $2\ell a_{r,k}\ge0$, while $(\Delta_k)_i\ge-D_\Delta$. Thus $z_{i,k+1}\ge z_{i,k}-2\ell A_N-D_\Delta$. Define $V_N=R+N(2\ell A_N+D_\Delta),\qquad \mathcal B_N=[-V_N,U_N]^d$. Its side length $L_N$ is $O(Ne^{\kappa N})$. Call it a near-cancellation event if any of the finitely many allowed polynomials $F_{i,S_0}$ has modulus below $\epsilon_N$; the union includes every allowed $S_0$, not just the one selected from the current weights. Let $\mathcal C_N$ be this union of sublevel sets restricted to $\mathcal B_N$. The side-length estimate gives the bound for $\mathcal B_N$. Applying Lemma~\ref{lem:normalizedSublevel} on this box and summing over the finite collection gives the bound for $\mathcal C_N$ when $d>1$. When $d=1$, each polynomial in the collection is a nonzero constant, so the sublevel sets are empty for large $N$. The construction uses the same constants for every sequence with translation bound $D_\Delta$.
\end{proof}

Fix a family of allowed sampling sequences and let $\mathcal W_k$ be its distinct prefixes of length $k$, including the empty prefix at $k=0$. Let $\mathcal E_N$ consist of initial ratios in $[-R,R]^d$ with $S_{r,N}\le W_N=M_0e^{\kappa N}$ along at least one sequence.

\begin{lemma}[Confinement before the first near cancellation]\label{lem:samplingEntry}
For a single prescribed sequence, let $\Psi_k$ denote its length-$k$ composition and let $\mathcal E_N$ be its small-decrease event. With the sets of Lemma~\ref{lem:samplingBoxes},
\[
 \mathcal E_N\subseteq\Psi_N^{-1}\mathcal B_N
 \ \cup\ \bigcup_{k=0}^{N-1}\Psi_k^{-1}\mathcal C_N.
\]
\end{lemma}
\begin{proof}\hspace{0.5em}
For a single prescribed sequence, on $\mathcal E_N$ the upper coordinate bound $U_N$ holds through time $N$. While $z_k$ stays in $\mathcal B_N\setminus\mathcal C_N$, the lemma bounds $a_{i,k}$ by $A_N$ for $i\ne r$. The ratio recurrence then gives $z_{i,k+1}\ge z_{i,k}-2\ell A_N-D_\Delta$. Summing this bound from the initial box gives $z_{i,k+1}\ge -V_N$ for $k+1\le N$, so the next iterate is also in $\mathcal B_N$. Thus either $z_N$ is in $\mathcal B_N$ or the first visit to $\mathcal C_N$ occurs at some $k<N$. The lower coordinate bound at that visit uses only the earlier updates. This gives the stated inclusion.
\end{proof}

\begin{lemma}[Terminal-box and first-entry measure bounds]\label{lem:samplingVolume}
Let $\mathcal B_N,\mathcal C_N$ be the common sets from Lemma~\ref{lem:samplingBoxes}. Then
\[
 \vol(\mathcal E_N)\le |\mathcal W_N|\rho^N\vol(\mathcal B_N)
   +\vol(\mathcal C_N)\sum_{k=0}^{N-1}|\mathcal W_k|\rho^k.
\]
If $|\mathcal W_k|\le C_he^{hk}$ and $e^h\rho<1$, the prefix sum is bounded independently of $N$. For any $c>0$, choosing $b$ sufficiently large gives $\vol(\mathcal C_N)\le Ce^{-cN}$, while $\vol(\mathcal B_N)\le CN^de^{\kappa dN}$. Constants are uniform over the sequences in the family.
\end{lemma}
\begin{proof}\hspace{0.5em}
For one prescribed sequence $w$, denote its event by $\mathcal E_N(w)$ and apply Lemma~\ref{lem:samplingEntry}. Taking measures in the inclusion and using \eqref{eq:nonautonomous} and $\sum_{k<N}\rho^k\le(1-\rho)^{-1}$ gives
\[
 \vol(\mathcal E_N(w))
 \le\rho^N\vol(\mathcal B_N)+(1-\rho)^{-1}\vol(\mathcal C_N).
\]
Substituting the two volume bounds from Lemma~\ref{lem:samplingBoxes} gives
\begin{equation}\label{eq:badvolume}
 \vol(\mathcal E_N(w))\le C\left[
 N^d(\rho e^{\kappa d})^N+
 N^{d-1}e^{(\kappa(d-1)-b/(d-1))N}\right]
\end{equation}
for $d>1$. In particular, these terms are summable when $\rho e^{\kappa d}<1$ and $b>\kappa(d-1)^2$. For $d=1$, the sublevel sets are empty for all sufficiently large $N$, and the first term suffices.

For each sequence of sampling choices $w$ of length $k$, the corresponding composition $\Psi_w$ is made up of translations of $T$ determined by $w$, and hence $\vol(\Psi_w^{-1}E)=\rho^k\vol(E)$. Taking the union over allowed sequences gives the following estimate, where $\vol^*$ denotes Lebesgue outer measure:
\begin{equation}\label{eq:wordvolume}
 \vol^*\left(\bigcup_{w\in\mathcal W_k}\Psi_w^{-1}E\right)
 \le|\mathcal W_k|\rho^k\vol(E).
\end{equation}
This union contains the inverse images produced by any history-dependent rule, even one that is not a function of the current log ratios alone. All translations and sampling factors have the common bounds assumed above.

For a family of sequences, the same first-entry inclusion is applied along each prefix before taking the union. The terminal contributions sum over $\mathcal W_N$, and first entries at time $k$ sum over $\mathcal W_k$. This proves the asserted measure bound. The terminal box term is bounded by $C_h e^{hN}\rho^N\vol(\mathcal B_N)$. At first entry time $k$, each prefix of sampling lists contributes at most $\rho^k\vol(\mathcal C_N)$. Summing over the possible prefixes gives $\vol(\mathcal C_N)\sum_{k=0}^{N-1} C_h e^{hk}\rho^k$, where the sum over $k$, not $N$, is bounded independently of $N$:
\[
 \sum_{k<N}C_h e^{hk}\rho^k
 \le\frac{C_h}{1-e^{h}\rho}<\infty.
\]

For $d>1$, the box lemma gives $\vol(\mathcal C_N)\le CN^{d-1}e^{[\kappa(d-1)-b/(d-1)]N}$. Choose $b$ so that the exponent is strictly less than $-c$; the extra polynomial factor is then absorbed into the constant. For $d=1$, these sets are empty for all sufficiently large $N$, and finitely many early terms can be absorbed into the same bound.
\end{proof}

\end{appendices}

\bibliography{sn-bibliography}

\end{document}